\documentclass{amsart}
\title[Coideal subalgebras of quantum $SL_2$ at roots of unity]
{Coideal subalgebras of quantum $SL_2$ \\ at roots of unity}
\author[K. Shimizu]{Kenichi Shimizu}
\address[K. Shimizu]{Department of Mathematical Sciences,
  Shibaura Institute of Technology \\
  307 Fukasaku, Minuma-ku, Saitama-shi, Saitama 337-8570, Japan.}
\email{kshimizu@shibaura-it.ac.jp}
\author[R. Sugitani]{Rei Sugitani}
\address[R. Sugitani]{Graduate School of Pure and Applied Sciences,
  University of Tsukuba,
  1-1-1 Tennodai, Tsukuba, Ibaraki 305-8571, Japan}
\email{rsugitani@math.tsukuba.ac.jp}

\usepackage{amsmath,amscd,amssymb}
\usepackage{bbm,dsfont,mathrsfs}
\usepackage{hyperref}
\usepackage{tikz-cd}

\usepackage{amsthm}
\numberwithin{equation}{section}
\newtheorem{C}{}[section]
\newtheorem{lemma}[C]{Lemma}
\newtheorem{theorem}[C]{Theorem}

\newtheorem{proposition}[C]{Proposition}
\theoremstyle{definition}
\newtheorem{definition}[C]{Definition}
\theoremstyle{remark}
\newtheorem{remark}[C]{Remark}
\newtheorem{example}[C]{Example}
\newtheorem*{claim*}{Claim}

\newcommand{\bfk}{\Bbbk}

\newcommand{\op}{\mathrm{op}}
\newcommand{\cop}{\mathrm{cop}}
\newcommand{\Span}{\mathop{\mathrm{span}}\nolimits}
\newcommand{\Hom}{\mathrm{Hom}}
\newcommand{\Rep}{\mathrm{Rep}}
\newcommand{\Mod}{\mathfrak{M}}
\newcommand{\Ker}{\mathrm{Ker}}

\newcommand{\dcsa}{\dagger}

\begin{document}

\begin{abstract}
  We classify right coideal subalgebras of the finite-dimensional quotient of the quantized enveloping algebra $U_q(\mathfrak{sl}_2)$ and that of the quantized coordinate algebra $\mathcal{O}_q(SL_2)$ at a root of unity $q$ of odd order.
  All those coideal subalgebras are described by generators and relations.
\end{abstract}

\maketitle

\tableofcontents

\section{Introduction}

Hopf algebras \cite{MR1243637} may be considered as an algebraic structure expressing symmetries in `non-commutative' geometry, and therefore one can ask what a homogeneous space over a Hopf algebra is.
After pioneering investigations of quotients of and torsors over Hopf algebras by Takeuchi \cite{MR549940,MR1271619}, Doi \cite{MR688207}, Schneider \cite{MR1098988} and Masuoka \cite{MR1140616}, it is now accepted to understand a `quantum homogeneous space' of a Hopf algebra $H$ as a coideal subalgebra of $H$ over which $H$ is faithfully flat.
M\"uller and Schneider \cite{MR1710737} have shown that many concrete examples of coideal subalgebras known at the time are indeed quantum homogeneous spaces in this sense.

Skryabin \cite{MR2286047} showed that a finite-dimensional Hopf algebra is free over every coideal subalgebra, and therefore the faithfully flatness condition is always satisfied in the finite setting.
Classifying coideal subalgebras of a given finite-dimensional Hopf algebra could be a fundamental problem such like a classification of subgroups of a given finite group. This problem has been settled in various setting: For example, coideal subalgebras of some classes of Kac algebras were classified in \cite{MR2900441,MR3270060,MR2774698}. Coideal subalgebras of the Taft algebra and the bosonization of the Fomin-Kirillov algebra were classified in \cite{MR4125589} and \cite{MR4207921}, respectively.

Coideal subalgebras of the quantized enveloping algebra $U_q(\mathfrak{g})$ of a semisimple Lie algebra $\mathfrak{g}$ have attracted attention in relation with quantum symmetric pairs and the reflection equation.
When the parameter $q$ is a root of unity, $U_q(\mathfrak{g})$ has a finite-dimensional quotient $\overline{U}_q(\mathfrak{g})$ as introduced by Lusztig \cite{Lus90}.
Although classification results are known for several classes of coideal subalgebras of $U_q(\mathfrak{g})$ or its Borel part \cite{MR2764869,MR2967250,MR3096611,MR2835327,MR2824522,MR2388323,MR2782598}, a complete classification of coideal subalgebras of $\overline{U}_q(\mathfrak{g})$ does not yet seem to be known even for $\mathfrak{g} = \mathfrak{sl}_2$.

In this paper, we develop a general method to give generators of coideal subalgebras of Hopf algebras obtained by lifting finite quantum linear spaces (Theorems~\ref{thm:generators-1} and \ref{thm:generators-2}) and, as an application, classify coideal subalgebras of the small quantum group $\overline{U}_q(\mathfrak{sl}_2)$ at a root of unity $q$ of odd order (Theorem~\ref{thm:CSA-Uq-sl-2}).
By this result and the bijection between the set of right coideal subalgebras of a finite-dimensional Hopf algebra and the set of those of the dual Hopf algebra established by Masuoka \cite{MR1157912} and Skryabin \cite{MR2286047}, we also classify right coideal subalgebras of the finite-dimensional quotient $\overline{\mathcal{O}}_q(SL_2)$ of the quantized coordinate algebra of $SL_2$ (Section~\ref{sec:CSA-Oq-SL2}).

It turns out that the list of coideal subalgebras of $\overline{U}_q(\mathfrak{sl}_2)$ at a root of unity is basically same as that of $U_q(\mathfrak{sl}_2)$ at generic $q$ given in \cite{2018arXiv180410007V}.
However, there are problems specific to the case of roots of unity.
For example, we consider the element $u = E + \alpha K F + \beta K$ of $U_q(\mathfrak{sl}_2)$ or $\overline{U}_q(\mathfrak{sl}_2)$, where $E$, $F$ and $K$ are the standard generators and $\alpha$ and $\beta$ are scalars.
The subalgebra generated by $u$ is in fact a coideal subalgebra.
The subalgebra $\langle u \rangle$ is just a polynomial algebra of one variable in $U_q(\mathfrak{sl}_2)$ at a generic parameter $q$, while $\langle u \rangle$ is finite-dimensional in $\overline{U}_q(\mathfrak{sl}_2)$ at a root of unity $q$.
In this paper, we actually give generators and determine defining relations of all coideal subalgebras of $\overline{U}_q(\mathfrak{sl}_2)$ and $\overline{\mathcal{O}}_q(SL_2)$ for a root of unity $q$ of odd order.
The minimal polynomial of $u$ (with parameters slightly changed) is given in Theorem \ref{thm:CSA-u-alpha-beta-min-pol}.

\subsection*{Organization of this paper}
This paper is organized as follows:
For a finite-dimensional Hopf algebra $H$, we denote by $\mathcal{C}(H)$ the set of right coideal subalgebras of $H$.
In Section \ref{sec:preliminaries}, after introducing basic notation, we recall Skryabin's freeness theorem for coideal subalgebras \cite{MR2286047} and the bijection between $\mathcal{C}(H)$ and $\mathcal{C}(H^*)$ due to Masuoka \cite{MR1157912} and Skryabin \cite{MR2286047}.
In this paper, we consider the setting where we are given two finite-dimensional Hopf algebras $H$ and $U$ and an isomorphism $\phi : H \to U^*$ of Hopf algebras. By composing the bijection between $\mathcal{C}(H)$ and $\mathcal{C}(H^*)$ and the map induced by $\phi$, we obtain a bijection
\begin{equation}
  \label{eq:intro-Masuoka-Skryabin-correspondence}
  \mathcal{C}(U) \to \mathcal{C}(H),
  \quad A \mapsto A^{\dcsa} := \phi^{-1}((U/\Ker(\varepsilon_A)U)^*),
\end{equation}
where $\varepsilon_A$ is the restriction of the counit of $U$ to $A$.
We give formulas and some properties of $A^{\dcsa}$ for the use of later sections.

In Section~\ref{sec:lifting}, we analyze generators of right coideal subalgebras of Hopf algebras $U(\mathcal{D})$ obtained by lifting finite quantum linear spaces.
The base field, say $\bfk$, is assumed to be algebraically closed and of characteristic zero from this section.
The Hopf algebra $U(\mathcal{D})$ is generated by a finite abelian group $\Gamma$ and skew-primitive elements $x_1, \cdots, x_{\theta}$, and has monomials of $x_1, \cdots, x_{\theta}$ as a basis over $\bfk \Gamma$.
As a natural extension of skew-derivations appearing in the study of Nichols algebras \cite{MR1913436,MR506406}, we construct linear operators $\partial_i$ ($i = 1, 2, \cdots, \theta$) on $U(\mathcal{D})$ behaving like a partial derivation with respect to $x_i$.
Every right coideal of $U(\mathcal{D})$ is stable under $\partial_i$ (Lemma~\ref{lem:partial-der-coideal}). This is a key observation to prove Theorem~\ref{thm:generators-1}, which states that a right coideal subalgebra $A$ of $U(\mathcal{D})$ is generated by the subspace
\begin{equation*}
  A \cap \Span \{ g, g x_i \mid g \in \Gamma, i = 1, \cdots, \theta \}
\end{equation*}
of $A$. With an eye toward practical applications, we also provide a theorem giving generators more explicitly (Theorem~\ref{thm:generators-2}). As a demonstration of this theorem, we classify right coideal subalgebras of Hopf algebras including the Taft algebra and $\mathrm{gr}\,\overline{U}_q(\mathfrak{sl}_2)$ (Examples \ref{ex:pointed-rank-1} and~\ref{ex:CSA-gr-Uq-sl2}).

In Section~\ref{sec:CSA-Uq-sl2}, we classify right coideal subalgebras of the Hopf algebra $\overline{U}_q(\mathfrak{sl}_2)$, where $q$ is a root of unity of odd order $N > 1$. Since $\overline{U}_q(\mathfrak{sl}_2)$ is an instance of $U(\mathcal{D})$, the classification is done by applying Theorem \ref{thm:generators-2}.
Consequently, the right coideal subalgebras of $\overline{U}_q(\mathfrak{sl}_2)$ are
\begin{gather*}
  \overline{U}_q(\mathfrak{sl}_2),
  \quad \langle K^r \rangle,
  \quad \langle K^r, E \rangle,
  \quad \langle K^r, \tilde{F} \rangle
  \quad (\text{$r$ is a positive divisor of $N$}), \\
  \langle E + \alpha \tilde{F} + \beta K \rangle
  \quad (\alpha, \beta \in \bfk, (\alpha, \beta) \ne (0,0)),
  \quad \langle \tilde{F} + \beta K \rangle \quad (\beta \in \bfk, \beta \ne 0), \\
  \quad \langle E + \lambda K, \tilde{F} + \mu K \rangle
  \quad (\lambda, \mu \in \bfk, (1-q^2) \lambda \mu = 1),
\end{gather*}
where $\tilde{F} = (q-q^{-1})K F$.
For each right coideal subalgebra of $\overline{U}_q(\mathfrak{sl}_2)$, we describe its algebra structure.
It turns out that the subalgebra $\langle E + \lambda K, \tilde{F} + \mu K \rangle$ with $\lambda \mu (1-q^2) = 1$ is isomorphic to the Taft algebra (Theorem~\ref{thm:CSA-v-lambda-w-mu-structure}).
The most difficult case is the subalgebra $\langle u_{\alpha,\beta} \rangle$ generated by $u_{\alpha,\beta} = E + \alpha \tilde{F} + \beta K$.
We give the minimal polynomial of $u_{\alpha,\beta}$ explicitly and determine when $\langle u_{\alpha,\beta} \rangle$ is semisimple as an algebra (Theorem~\ref{thm:CSA-u-alpha-beta-min-pol}).

In Section~\ref{sec:CSA-Oq-SL2}, we classify right coideal subalgebras of the quotient $\overline{\mathcal{O}}_q(SL_2)$ of the quantized coordinate algebra.
Since $H = \overline{\mathcal{O}}_q(SL_2)$ is dual to $U = \overline{U}_q(\mathfrak{sl}_2)$, a complete list of the right coideal subalgebras of $\overline{\mathcal{O}}_q(SL_2)$ is obtained from the list of those of $\overline{U}_q(\mathfrak{sl}_2)$ via the bijection \eqref{eq:intro-Masuoka-Skryabin-correspondence}.
We also determine generators and defining relations of right coideal subalgebras of $\overline{\mathcal{O}}_q(SL_2)$. For example, we have
\begin{gather*}
  \langle K^r \rangle^{\dcsa} = \langle a^{N/r}, a^{-1} b, a c \rangle
  = \langle d^{N/r}, b d, c d^{-1} \rangle,
  \quad \langle K^r, E \rangle^{\dcsa} = \langle d^{N/r}, c d^{-1} \rangle, \\
  \langle E + \alpha K \rangle^{\dcsa} = \langle c d^{-1}, d^2 + (q-q^{-1}) \alpha b d \rangle \quad (\alpha \in \bfk),
\end{gather*}
etc., where $a$, $b$, $c$ and $d$ are the generators of $\overline{\mathcal{O}}_q(SL_2)$ (Theorem~\ref{thm:CSA-Oq-SL2-generators}).
Although we do not have handy formulas of generators of $\langle E + \alpha \tilde{F} + \beta K \rangle^{\dcsa}$ and $\langle E + \lambda K, \tilde{F} + \mu K \rangle^{\dcsa}$ in the general case, we express their generators in terms of integrals in coideal subalgebras (Theorems~\ref{thm:CSA-Oq-SL2-generators-2} and ~\ref{thm:CSA-Oq-SL2-generators-3}).

In Section~\ref{sec:remarks}, we add some supplementary remarks.
First, we give lists of normal coideal subalgebras of $\overline{U}_q(\mathfrak{sl}_2)$ and $\overline{\mathcal{O}}_q(SL_2)$.
Second, we give a negative answer to a question raised in \cite{MR4125589} concerning the number of orbits of coideal subalgebras by Hopf automorphisms.
Third, we point out that the condition `$S^2(A) \subset A$' cannot be dropped from the following Maschke type theorem due to Koppinen: If a coideal subalgebra $A$ of a finite-dimensional Hopf algebra possesses a total integral and $S^2(A) \subset A$, then $A$ is semisimple \cite{MR1199682}.
Finally, we give examples of simple algebras in $\Rep(\overline{U}_q(\mathfrak{sl}_2))$.

\section{Preliminaries}
\label{sec:preliminaries}

\subsection{Convention}
\label{subsec:convention}

The standard notation for intervals in analysis will be diverted to mean intervals of integers. For example, given two integers $a$ and $b$, we have
\begin{equation*}
  [a, b] = \{ x \in \mathbb{Z} \mid a \le x \le b \} \quad \text{and} \quad
  [a, b) = \{ x \in \mathbb{Z} \mid a \le x < b \}.
\end{equation*}

Throughout this paper, we work over a field $\bfk$.
No additional assumptions are imposed on $\bfk$ in this section, however, $\bfk$ will be an algebraically closed field of characteristic zero after Section \ref{sec:lifting}.
Given a vector space $X$ (over $\bfk$), we denote by $X^*$ the dual space of $X$.
For elements $f \in X^*$ and $x \in X$, we often write $f(x)$ as $\langle f, x \rangle$.
Unless otherwise noted, the symbol $\otimes$ means the tensor product over $\bfk$.
By a (co)algebra, we always mean a (co)associative (co)unital (co)algebra over $\bfk$.
The comultiplication and the counit of a coalgebra are denoted by $\Delta$ and $\varepsilon$, respectively.
To express the comultiplication, we adopt the Sweedler notation, such as
\begin{equation*}
  \Delta(c) = c_{(1)} \otimes c_{(2)}, \quad
  \Delta(c_{(1)}) \otimes c_{(2)}
  = c_{(1)} \otimes c_{(2)} \otimes c_{(3)}
  = c_{(1)} \otimes \Delta(c_{(2)}).
\end{equation*}

Let $C$ be a coalgebra. Given a right $C$-comodule $M$, we denote the right coaction of $C$ on $M$ by $\delta_M : M \to M \otimes C$ and express it as $\delta_M(m) = m_{(0)} \otimes m_{(1)}$ like Sweedler's notation. We note that $C^*$ has a natural structure of an algebra and acts on a right $C$-comodule $M$ from the left by the action defined by $f \rightharpoonup m = m_{(0)} \langle f, m_{(1)} \rangle$ for $m \in M$ and $f \in C^*$. The algebra $C^*$ acts on a left $C$-comodule from the right in an analogous way.

\newcommand{\Grp}{\mathrm{G}}
\newcommand{\Prim}{\mathrm{P}}

We refer the reader to Montgomery's book \cite{MR1243637} for basics on Hopf algebras.
The antipode of a Hopf algebra $H$ is denoted by $S$.
We define and denote the group of {\em grouplike elements} of $H$ by
\begin{equation*}
  \Grp(H) = \{ g \in H \mid \text{$\Delta(g) = g \otimes g$ and $\varepsilon(g) = 1$} \}.
\end{equation*}
According to \cite{MR1243637}, we also define
\begin{equation*}
  \Prim_{g, h}(H) = \{ p \in H \mid \Delta(p) = p \otimes g + h \otimes p \}
  \quad (g, h \in \Grp(H))
\end{equation*}
and call an element of this space a $(g, h)$-{\em skew primitive element}\footnote{The space of $(g, h)$-skew ptimieive elements of $H$ is denoted by $P_{h,g}(H)$ in \cite{MR1659895}.}.

\subsection{Freeness over coideal subalgebras}

Let $H$ be a Hopf algebra. We recall that a {\em right coideal} of $H$ is a subspace $V \subset H$ such that $\Delta(V) \subset V \otimes H$.
A {\em right coideal subalgebra} of $H$ is a subalgebra of $H$ that is simultaneously a right coideal of $H$.
A {\em left coideal subalgebra} is defined in an analogous way.
From now on, unless otherwise noted, a coideal subalgebra means a right coideal subalgebra.

Given a coideal subalgebra $A$ of $H$, a {\em right $(H, A)$-Hopf module} \cite{MR549940} is a right $A$-module $M$ endowed with a structure of a right $H$-comodule such that the coaction $M \to M \otimes H$ is right $A$-linear. We denote the category of $(H, A)$-Hopf modules by $\Mod^H_A$.
Provided that $H$ is finite-dimensional, Skryabin \cite{MR2286047} showed that every object of $\Mod^H_A$ is free as a right $A$-module. We refer this result {\em Skryabin's freeness theorem}.

\subsection{Coideal subalgebras of the dual Hopf algebra}

Given a Hopf algebra $H$, we denote by $\mathcal{C}(H)$ the set of coideal subalgebras of $H$. We note that $\mathcal{C}(H)$ is in fact a lattice with respect to the inclusion order. The meet $A \wedge B$ for $A, B \in \mathcal{C}(H)$ is just the intersection of $A$ and $B$, and the join $A \vee B$ is the subalgebra generated by $A$ and $B$.

We assume that $H$ is finite-dimensional.
As an application of his freeness theorem, Skryabin \cite{MR2286047} showed that every coideal subalgebra of $H$ is a Frobenius algebra.
By combining this result with Masuoka's result \cite[Proposition 2.10]{MR1157912}, a bijection between $\mathcal{C}(H)$ and $\mathcal{C}(H^*)$ is established as follows:
Given $A \in \mathcal{C}(H)$, we set $A^{+} = A \cap \Ker(\varepsilon)$.
Then $H/A^{+}H$ is a right $H$-module coalgebra as a quotient of $H$.
We note that the dual space $M^*$ of a right $H$-module $M$ is a left $H$-module by the action given by $(h \rightharpoonup f)(m) = f(m h)$ for $f \in H^*$, $h \in H$ and $m \in M$.
The left $H$-module $(H/A^{+}H)^*$, regarded as an $H$-submodule of $H^*$, is a right coideal subalgebra of $H^*$. The bijection is given by
\begin{equation}
  \label{eq:Masuoka-Skryabin-correspondence}
  \mathcal{C}(H) \to \mathcal{C}(H^*),
  \quad A \mapsto (H/A^{+}H)^*.
\end{equation}

Let $A$ be a coideal subalgebra of $H$. Then $H$ is a right comodule over the quotient coalgebra $D := H/A^{+}H$. Skryabin \cite[Theorem 6.1]{MR2286047} showed that there is an isomorphism $H \cong A \otimes D$ of left $A$-modules as well as an isomorphism of right $D$-comodules. In particular, we have the following relation:

\begin{lemma}
  \label{lem:dual-coideal-sub-1}
  For $A \in \mathcal{C}(H)$, we have
  \begin{equation*}
    \dim(A) \cdot \dim(H/A^{+}H) = \dim(H).
  \end{equation*}
\end{lemma}

Let $U$ and $H$ be finite-dimensional Hopf algebras.
Suppose that we are given an isomorphism $\phi : H \to U^*$ of Hopf algebras.
By composing the bijection~\eqref{eq:Masuoka-Skryabin-correspondence} for $U$ and the bijection $\mathcal{C}(U^*) \to \mathcal{C}(H)$ induced by $\phi$, we have a bijection
\begin{equation}
  \label{eq:Masuoka-Skryabin-correspondence-2}
  \mathcal{C}(U) \to \mathcal{C}(H),
  \quad A \mapsto A^{\dcsa} := \phi^{-1}((U/A^{+}U)^*).
\end{equation}

We note that the bijection \eqref{eq:Masuoka-Skryabin-correspondence-2} reverses the order. Thus it is in fact an isomorphism $\mathcal{C}(U) \cong \mathcal{C}(H)^{\op}$ of partially ordered sets. Hence, in particular, we have
\begin{equation}
  \label{eq:Masuoka-Skryabin-correspondence-3}
  (A \cap B)^{\dcsa} = A^{\dcsa} \vee B^{\dcsa}, \quad
  (A \vee B)^{\dcsa} = A^{\dcsa} \cap B^{\dcsa}
  \quad (A, B \in \mathcal{C}(U)).
\end{equation}

By Lemma \ref{lem:dual-coideal-sub-1} and the definition of $A^{\dcsa}$, we have
\begin{equation}
  \label{eq:dual-coideal-sub-dim}
  \dim(A^{\dcsa}) \cdot \dim(A) = \dim(H).
\end{equation}

There is a right action of $U$ on $H$ defined and denoted by
\begin{equation*}
  h \leftharpoonup u = (h_{(1)}, u) h_{(2)}
  \quad (u \in U, h \in h),
\end{equation*}
where $(h, u) = \langle \phi(h), u \rangle$ is the Hopf pairing induced by $\phi$.
With this notation, the coideal subalgebra $A^{\dcsa}$ is expressed as follows:

\begin{lemma}
  \label{lem:dual-coideal-sub-2}
  For $A \in \mathcal{C}(U)$, we have
  \begin{equation}
    \label{eq:A-dual-formula}
    A^{\dcsa} = \{ h \in H \mid
    \text{$h \leftharpoonup a = \varepsilon(a) h$ for all $a \in A$} \}.
  \end{equation}
\end{lemma}
\begin{proof}
  Since $(-,-)$ is a Hopf pairing, we have
  \begin{equation}
    \label{eq:dual-coideal-sub-2-proof-1}
    (h \leftharpoonup u, w)
    = (h_{(1)}, u) (h_{(2)}, w)
    = (h, u w)
    \quad (h \in H, u, w \in U).
  \end{equation}
  Let $f \in A^{\dcsa}$ and $a \in A$.
  Then, by the definition of $A^{\dcsa}$, we have $\phi(f)(x) = 0$ for all elements $x \in A^{+}L$.
  Since $a^{+} := a - \varepsilon(a)1$ belongs to $A^{+}$, we have
  \begin{equation*}
    (f \leftharpoonup a, u)
    - (\varepsilon(a) f, u)
    = (f \leftharpoonup a^{+}, u)
    \mathop{=}^{\eqref{eq:dual-coideal-sub-2-proof-1}} (f, a^{+} u) = 0.
  \end{equation*}
  By the non-degeneracy of $(-, -)$, we have $f \leftharpoonup a = \varepsilon(a) f$. Thus we have proved the `$\subset$' part of \eqref{eq:A-dual-formula}.
  To prove the converse inclusion, we let $f \in H$ be an element such that $h \leftharpoonup a = \varepsilon(a) f$ for all $a \in A$.
  For $a \in A^{+}$ and $u \in U$, we have
  \begin{equation*}
    \langle \phi(f), a u \rangle = (f, a u)
    \mathop{=}^{\eqref{eq:dual-coideal-sub-2-proof-1}}
    (f \leftharpoonup a, u)
    = (\varepsilon(a) f, u) = 0
  \end{equation*}
  since $\varepsilon(a) = 0$. Thus $f \in \phi^{-1}((U/A^{+}U)^*)$, as desired. The proof is done.
\end{proof}

\subsection{Integrals of coideal subalgebras}

Let $H$ be a finite-dimensional Hopf algebra, and let $A$ be a coideal subalgebra of $H$. A {\em right integral} of $A$ is an element $\Lambda \in A$ such that the equation $\Lambda a = \varepsilon(a) \Lambda$ holds for all $a \in A$.
Since $A$ is a Frobenius algebra \cite{MR2286047}, a non-zero right integral of $A$ always exists and is unique up to scalar multiple.

\begin{lemma}
  \label{lem:dual-coideal-sub-3}
  Retain the notation from Lemma~\ref{lem:dual-coideal-sub-2}.
  For $A \in \mathcal{C}(U)$, the corresponding coideal subalgebra $A^{\dcsa}$ of $H$ is given by
  \begin{equation}
    \label{eq:A-dual-formula-2}
    A^{\dcsa} = H \leftharpoonup \Lambda,
  \end{equation}
  where $\Lambda$ is a non-zero right integral of $A$.
\end{lemma}
\begin{proof}
  The action of $U$ on $H$ is defined so that $\phi : H \to U^*$ is an isomorphism of right $U$-modules.
  Since $U$ is a Frobenius algebra, $U^*$ is isomorphic to $U$ as a right $U$-module.
  By this observation and Skryabin's freeness theorem, we see that $H$ is free as a right $A$-module.
  Now, for a right $A$-module $M$, we set
  \begin{equation*}
    I(M) = \{ m \in M \mid \text{$m a = \varepsilon(a) m$ for all $a \in A$} \}.
  \end{equation*}
  Then we have $I(A) = \bfk \Lambda = A \Lambda$ by the fact that a non-zero right integral of $A$ is unique up to scalar multiple. Hence, for a free right $A$-module $F$, we have $I(F) = F \Lambda$. The proof is done by letting $F = H$.
\end{proof}

\section{Lifting of finite quantum linear spaces}
\label{sec:lifting}

\subsection{Lifting of finite quantum linear spaces}

In this section, $\bfk$ is assumed to be an algebraically closed field of characteristic zero.
A quantum linear space is a Nichols algebra of diagonal type with Dynkin diagram $A_1 \times \cdots \times A_1$.
In an earlier stage of the investigation of pointed Hopf algebras and Nichols algebras, Andruskiewitsch and Schneider \cite{MR1659895} classified all pointed Hopf algebras obtained by lifting of finite quantum linear spaces.
The aim of this section is to give a convenient set of generators of coideal subalgebras of these Hopf algebras.

We first introduce the finite-dimensional pointed Hopf algebra $U(\mathcal{D})$ with reference to Andruskiewitsch and Schneider \cite{MR1659895}.
Let $\theta$ be a positive integer, which we call the {\em rank}.
The Hopf algebra $U(\mathcal{D})$ is constructed from a datum
\begin{equation}
  \label{eq:datum}
  \mathcal{D} = (\Gamma, (g_i)_{1 \le i \le \theta},
  (\chi_i)_{1 \le i \le \theta},
  (\lambda_{i j})_{1 \le i < j \le \theta},
  (\mu_i)_{1 \le i \le \theta})
\end{equation}
described as follows: $\Gamma$ is a finite abelian group with identity element $e$, $g_i$ is an element of $\Gamma$, and $\chi_i$ is an element of the group $\hat{\Gamma} := \Hom(\Gamma, \bfk^{\times})$ such that
\begin{equation}
  \label{eq:datum-disconnectedness}
  \chi_i(g_i) \ne 1, \quad \chi_i(g_j) \chi_j(g_i) = 1
  \quad (i, j \in [1, \theta], i \ne j).
\end{equation}
$\lambda_{i j}$ is an element of $\bfk$ such that
\begin{equation}
  \label{eq:datum-linking}
  \text{$\lambda_{i j} = 0$
    if ($g_i g_j = e$ or $\chi_i \chi_j \ne \varepsilon$)},
\end{equation}
where $\varepsilon$ is the identity element of $\hat{\Gamma}$.
$\mu_i$ is an element of $\{ 0, 1 \}$ such that
\begin{equation}
  \label{eq:datum-root-vec}
  \text{$\mu_i = 0$
    if ($g_i^{N_i} = 1$ or $\chi_i^{N_i} \ne \varepsilon$)},
\end{equation}
where $N_i$ is the order of $\chi_i(g_i)$.

\begin{definition}
  Given a datum $\mathcal{D}$ as in \eqref{eq:datum}, the Hopf algebra $U(\mathcal{D})$ is defined as follows: As an algebra, it is generated by the symbols $u_g$ ($g \in \Gamma$) and $x_i$ ($i = 1, \cdots, \theta$) subject to the defining relations
  \begin{gather}
    \label{eq:uD-defining-relations-1}
    u_e = 1, \quad u_g u_h = u_{g h}, \\
    \label{eq:uD-defining-relations-2}
    u_g x_i = \chi_i(g) x_i u_g, \\
    \label{eq:uD-defining-relations-3}
    x_j x_i = \chi_i(g_j) x_i x_j + \lambda_{i j}(1 - u_{g_i} u_{g_j}), \\
    \label{eq:uD-defining-relations-4}
    x_i^{N_i} = \mu_i(1 - u_{g_i}^{N_i})
  \end{gather}
  for $g, h \in \Gamma$ and $i, j \in [1, \theta]$ with $i < j$.
  The group algebra $\bfk \Gamma$ is a subalgebra of $U(\mathcal{D})$ by identifying $g \in \Gamma$ with $u_g \in U(\mathcal{D})$.
  Thus, in what follows, we write $u_g$ ($g \in \Gamma$) simply as $g$.
  The Hopf algebra structure of $U(\mathcal{D})$ is determined by
  \begin{equation}
    \label{eq:uD-comultiplication-1}
    \Delta(g) = g \otimes g \quad (g \in \Gamma),
    \quad \Delta(x_i) = x_i \otimes g_i + 1 \otimes x_i
    \quad (i = 1, \cdots, \theta).
  \end{equation}
\end{definition}

The Hopf algebra $\mathcal{A}(\Gamma, \mathcal{R}, \mathcal{D})$ of \cite{MR1659895} is identical to $U(\mathcal{D})^{\cop}$, where $C^{\cop}$ for a coalgebra $C$ means the opposite coalgebra of $C$. The comultiplication opposite to \cite{MR1659895} is convenient for us since, for example, the subalgebra generated by $x_i$ is a right coideal subalgebra of $U(\mathcal{D})$.
According to \cite{MR1659895}, the set
\begin{equation}
  \label{eq:uD-basis}
  \{ g x_1^{i_1} x_2^{i_2} \cdots x_{\theta}^{i_{\theta}} \mid \text{$g \in \Gamma$ and $i_k \in [0, N_k)$ for all $k \in [1, \theta]$} \}
\end{equation}
is a basis of $U(\mathcal{D})$.
Thus $\bfk \Gamma$ is a Hopf subalgebra of $U(\mathcal{D})$ and the monomials of $x_i$'s are a basis of $U(\mathcal{D})$ over $\bfk \Gamma$.

\subsection{Generating set of a coideal subalgebra}
We fix a datum $\mathcal{D}$ as in \eqref{eq:datum} and use the same notation as above.
One of main results of this section is:

\begin{theorem}
  \label{thm:generators-1}
  A coideal subalgebra $A \subset U(\mathcal{D})$ is generated by
  \begin{equation}
    \label{eq:coideal-sub-generators-1}
    A \cap \Span\{ g, g x_i \mid g \in \Gamma, i \in [1, \theta] \}.
  \end{equation}
\end{theorem}

This simple-look theorem may not be helpful to classify coideal subalgebras of $U(\mathcal{D})$.
For practical applications, we also provide a theorem on coideal subalgebras of $U(\mathcal{D})$ which gives generators more explicitly.
To explain our result, we introduce some notations:
Given a subgroup $G$ of $\Gamma$, we introduce the equivalence relation $\sim_G$ on the index set $[1, \theta]$ by
\begin{equation*}
  i \sim_G j \iff
  \text{$g_i \equiv g_j$ (mod $G$)
    and $\chi_i \equiv \chi_j$ (mod $G^{\perp}$)},
\end{equation*}
where $G^{\perp} = \{ \chi \in \hat{\Gamma} \mid \text{$\chi(g) = 1$ for all $g \in G$} \}$.
Let $\mathfrak{J}$ be an equivalence class of $\sim_G$, and and write it as $\mathfrak{J} = \{ i_1, \cdots, i_n \}$, where $i_1 < \cdots < i_n$. We define
\begin{equation*}
  \xi_{\mathfrak{J}}(c_1, \cdots, c_n, a) = g_{i_1}^{} \sum_{s = 1}^{n} c_{s} g_{i_s}^{-1} x_{i_s}^{} + a g_{i_1}^{}
  \quad (c_1, \cdots, c_n, a \in \bfk).
\end{equation*}
We note that the element $\xi := \xi_{\mathfrak{J}}(c_1, \cdots, c_n, a)$ satisfies
\begin{equation}
  \label{eq:comultplication-xi}
  \Delta(\xi)
  = \xi \otimes g_{i_1}^{}
  + \sum_{s = 1}^n c_s g_{i_1}^{} g_{i_s}^{-1} \otimes x_{i_s}
  \in (\bfk \xi + \bfk G) \otimes U(\mathcal{D}).
\end{equation}
Given a matrix $C$ of $n \times (n + 1)$, we define
\begin{equation*}
  X_{\mathfrak{J}}(C) = \{ \xi_{\mathfrak{J}}(c_{r,1}, \dotsc, c_{r,n}, c_{r,n+1}) \mid r \in [1, n] \},
\end{equation*}
where $c_{r s}$ is the $(r, s)$ entry of $C$.
We say that the matrix $C$ is a {\em reduced datum} for the equivalence class $\mathfrak{J}$ if the following three conditions (RD1)--(RD3) are satisfied:
\begin{itemize}
\item [(RD1)] The matrix $C$ has the reduced row echelon form.
\item [(RD2)] The last column of $C$ is zero if $\chi_j \not \equiv \varepsilon \pmod{G^{\perp}}$.
\item [(RD3)] The matrix $C$ has the same rank as the submatrix of $C$ obtained by removing the last column from $C$. 
\end{itemize}

\begin{theorem}
  \label{thm:generators-2}
  Let $A$ be a coideal subalgebra of $U(\mathcal{D})$, and let $\mathfrak{J}_1, \cdots, \mathfrak{J}_m$ be the set of the equivalence classes of $\sim_G$, where $G = \Gamma \cap A$. Then there are reduced data $C_k$ for $\mathfrak{J}_k$ for each $k$ such that the set
  \begin{equation}
    \label{eq:coideal-sub-generators-2}
    G \cup X_{\mathfrak{J}_1}(C_1) \cup \cdots \cup X_{\mathfrak{J}_k}(C_k)
  \end{equation}
  generates $A$ as an algebra.
\end{theorem}

Theorem \ref{thm:generators-1} is an immediate consequence of Theorem \ref{thm:generators-2}.
Nevertheless, we will give a direct proof of Theorem \ref{thm:generators-1} in Subsection~\ref{subsec:proof-generators-1} to demonstrate the basic strategy. By refining the argument of Subsection~\ref{subsec:proof-generators-1}, we will prove Theorem \ref{thm:generators-2} in Subsection~\ref{subsec:proof-generators-2}.

We note that Theorem \ref{thm:generators-2} does not give a classification of coideal subalgebras of $U(\mathcal{D})$.
Given a subgroup $G$ and reduced data $C_k$ ($k \in [1, m]$), we can consider the subalgebra, say $A$, generated by \eqref{eq:coideal-sub-generators-2}.
We see that $A$ is a coideal subalgebra by \eqref{eq:comultplication-xi}.
However, as we will see in the next subsection below, $A \cap \Gamma$ does not coincide with $G$ in general.

\subsection{Examples}
\label{subsec:examples-of-thm}

Before we go into the proof of Theorems \ref{thm:generators-1} and \ref{thm:generators-2}, we demonstrate how these theorems work.
For $q \in \bfk^{\times}$ and $m \in \mathbb{Z}_{\ge 0}$, we define
\begin{equation*}
  (m)_{q} = 1 + q + \cdots + q^{m-1}
  \quad \text{and} \quad
  (m)_{q}! = (1)_{q} (2)_{q} \cdots (m)_{q}
  \quad (m > 0)
\end{equation*}
with convention $(0)_q = 0$ and $(0)_{q}! = 1$.
The following $q$-binomial formula \cite[VI.1.9]{MR1321145} will be used extensively:
For a non-negative integer $n$ and two elements $X$ and $Y$ of an algebra such that $Y X = q X Y$, we have
\begin{equation*}
  (X + Y)^n = \sum_{i = 0}^{n} \binom{n}{i}_{\!\!q} X^i Y^{n-i},
  \quad \text{where} \quad
  \binom{n}{i}_{\!\!q} = \frac{(n)_{q}!}{(i)_{q}!(n-i)_{q}!}.
\end{equation*}

\begin{example}
  \label{ex:pointed-rank-1}
  Let $N$ and $m$ be integers with $1 \le m < N$,
  let $\omega \in \bfk$ be an $N$-th root of unity such that $\omega^m \ne 1$,
  and let $\mu \in \{ 0, 1 \}$.
  We assume $\mu = 0$ if ($N \mid m N_1$ or $\omega^{N_1} \ne 1$), where $N_1 = \mathrm{ord}(\omega^m)$.
  With the use of Theorem \ref{thm:generators-2}, we classify coideal subalgebras of the Hopf algebra $H = U(\mathcal{D})$ associated to the datum \eqref{eq:datum} with $\theta = 1$, $\Gamma = \langle g \mid g^{N} = e \rangle$, $g_1 = g^m$, $\chi_1(g^a) = \omega^a$ ($a \in \mathbb{Z}$) and $\mu_1 = \mu$.
  By definition, $H$ is generated by a grouplike element $g$ and a $(g^m, 1)$-skew primitive element $x$ ($=x_1$) subject to $g^{N} = 1$, $g x = \omega x g$ and $x^{N_1} = \mu(1 - g^{m N_1})$.
  Let $A$ be a coideal subalgebra of $H$, and set $G = A \cap \Gamma$.
  Since $\Gamma$ is a cyclic group, we have $G = \langle g^r \rangle$ for some positive divisor $r$ of $N$.
  We note that the condition $\chi_1 \equiv \varepsilon \pmod{G^{\perp}}$ is equivalent to $\omega^{r} = 1$.
  \begin{itemize}
  \item We first discuss the case where $\omega^r = 1$.
    In this case, a reduced datum is either of the form $(0 \ 0)$ or $(1 \ \lambda)$ for some $\lambda \in \bfk$. Correspondingly, $A$ is either one of $A' := \langle g^{r} \rangle$ or $A''_{\lambda} := \langle g^{r}, v_{\lambda} \rangle$, where $v_{\lambda} = x + \lambda g^{m}$.
    It is obvious that $A' \cap \Gamma = G$.
    The coideal subalgebra $A''_{\lambda}$ satisfies $A''_{\lambda} \cap \Gamma = G$ if and only if the following condition holds:
    \begin{equation}
      \label{eq:pointed-rank-1-eq-2}
      \text{$\lambda^{N_1} = \mu$ \quad or \quad $r \mid m N_1$}
    \end{equation}
    Indeed, by the $q$-binomial formula, we have
    \begin{equation}
      \label{eq:pointed-rank-1-eq-1}
      v_{\lambda}^{N_1} = x^{N_1} + (\lambda g^{m})^{N_1} = \mu + (\lambda^{N_1} - \mu) g^{m N_1}.
    \end{equation}
    We suppose that \eqref{eq:pointed-rank-1-eq-2} holds.
    It follows from \eqref{eq:pointed-rank-1-eq-1} that the set
    \begin{equation*}
      \{ g^{r k} v_{\lambda}^{\ell} \mid k \in [0, N/r), \ell \in [0, N_1) \}
    \end{equation*}
    spans the vector space $A''_{\lambda}$.
    By using the basis \eqref{eq:uD-basis} of $H$, we see that the set $\{ v_{\lambda}^{\ell} \mid \ell \in [0, N_1) \}$ is linearly independent over $\bfk \Gamma$.
    Hence we have
    \begin{equation*}
      A''_{\lambda} \cap \bfk \Gamma = \Span \{ g^{r k} \mid k \in [0, N/r) \} = \bfk G,
    \end{equation*}
    and the equation $A''_{\lambda} \cap \Gamma = G$ is easily obtained from this.
    On the other hand, if \eqref{eq:pointed-rank-1-eq-2} does not hold, then $g^{m N_1}$ belongs to $A''_{\lambda} \cap \Gamma$ by \eqref{eq:pointed-rank-1-eq-1} but does not to $\langle g^r \rangle$. Hence $A''_{\lambda} \cap \Gamma \ne G$. The claim has been verified.
  \item Next, we discuss the case where $\omega^r \ne 1$.
    Since a reduced datum is either $(0 \ 0)$ or $(1 \ 0)$ in this case, $A$ is either $A' = \langle g^{r} \rangle$ or $A''_{0} = \langle g^r, x \rangle$.
    It is obvious that $A' \cap \Gamma = G$.
    By the same argument as above, $A''_{0} = \langle g^{r}, x \rangle$ satisfies $A''_{0} \cap \Gamma = G$ if and only if \eqref{eq:pointed-rank-1-eq-2} with $\lambda = 0$ is satisfied.
  \end{itemize}
  By summarizing the discussion so far, we conclude that the coideal subalgebras of $H$ are $\langle g^r \rangle$ and $\langle g^{r}, v_{\lambda} \rangle$, where $r$ is a positive divisor of $N$ and $\lambda$ is an element of $\bfk$ satisfying \eqref{eq:pointed-rank-1-eq-2} and ($\omega^r \ne 1 \Rightarrow \lambda = 0$).
  Here are some sub-examples:
  \begin{enumerate}
  \item Assume $m = 1$, $\mathrm{ord}(\omega) = N$ and $\mu = 0$.
    The above discussion yields that the coideal subalgebras of $H$ are $\langle g^r \rangle$, $\langle g^r, x \rangle$ and $\langle v_{\lambda} \rangle$, where $r$ is a positive divisor of $N$ and $\lambda \in \bfk^{\times}$ (the case where $\lambda = 0$ is excluded since $\langle v_{\lambda} \rangle = \langle g^r, x \rangle$ if $\lambda = 0$ and $r = N$).
    This recovers the classification result of coideal subalgebras of the Taft algebra given in \cite{MR4125589}.
  \item Assume $N = p^n$, $\mathrm{ord}(\omega) = p^{k}$, $m = 1$ and $\mu = 1$, where $p$ is a prime number and $n$ and $k$ are integers such that $0 < k < n$.
    The coideal subalgebras of $H$ are $\langle g^{p^i} \rangle$ and $\langle g^{p^{i}}, x + \lambda_i g \rangle$ for $i \in [0, n]$, where $\lambda_i$ is an element of $\bfk$ such that $\lambda_i = 0$ if $i < k$ and $\lambda_i^{N_i} = 1$ if $i > k$ (the parameter $\lambda_k$ can be arbitrary).
  \end{enumerate}
\end{example}

\begin{example}
  \label{ex:CSA-gr-Uq-sl2}
  We fix an odd integer $N > 1$ and a root of unity $q$ of order $N$.
  There is a Hopf algebra $H$ generated by a grouplike element $g$ and $(g,1)$-skew primitive elements $x_1$ and $x_2$ subject to $g^N = 1$, $g x_1 = q^{2} x_1 g$, $g x_2 = q^{-2} x_2 g$, $x_1^N = x_2^N = 0$ and $x_2 x_1 = q^2 x_1 x_2$.
  The Hopf algebra $H$ is isomorphic to $U(\mathcal{D})$, where $\mathcal{D}$ is a datum \eqref{eq:datum} with $\theta = 2$, $\Gamma = \langle g \mid g^N = e \rangle$, $g_1 = g_2 = g$, $\chi_1(g) = q^2$, $\chi_2(g) = q^{-2}$ and $\lambda_{1 2} = \mu_1 = \mu_2 = 0$. We also note that $H$ is isomorphic to the graded Hopf algebra $\mathrm{gr}\,\overline{U}_q(\mathfrak{sl}_2)$, where the filtration on $\overline{U}_q(\mathfrak{sl}_2)$ is given by $\deg(E) = \deg(F) = 1$ and $\deg(K) = 0$ (see Subsection \ref{subsec:Uq-sl2}).
  We classify coideal subalgebras of $H$. Let $A$ be a coideal subalgebra of $U(\mathcal{D})$, and set $G = A \cap \Gamma$.
  \begin{itemize}
  \item Suppose that $G = \{ e \}$.
    Then we have $1 \sim_G 2$.
    A reduced datum $C$ for the equivalence class $\{ 1, 2 \}$ is either of the form
    \begin{equation*}
      \begin{pmatrix} 0 & 0 & 0 \\ 0 & 0 & 0 \end{pmatrix}, \quad
      \begin{pmatrix} 1 & \alpha & \beta \\ 0 & 0 & 0 \end{pmatrix}, \quad
      \begin{pmatrix} 0 & 1 & \beta \\ 0 & 0 & 0 \end{pmatrix} \quad \text{or} \quad
      \begin{pmatrix} 1 & 0 & \alpha \\ 0 & 1 & \beta \end{pmatrix},
    \end{equation*}
    for some $\alpha, \beta, \lambda, \mu \in \bfk$.
    Correspondingly, $A$ is either one of $\bfk$, $\langle u_{\alpha, \beta} \rangle$, $\langle w_{\beta} \rangle$ or $\langle v_{\alpha}, w_{\beta} \rangle$, where $u_{\alpha,\beta} := x_1 + \alpha x_2 + \beta g$,
    $v_{\alpha} = x_1 + \alpha g$ and $w_{\beta} = x_2 + \beta g$.
    It is trivial that the intersection of $\bfk$ with $\Gamma$ is $\{ e \}$.
    The intersection of $\langle u_{\alpha,\beta} \rangle$ with $\Gamma$ is also equal to $\{ e \}$ since it is commutative and every non-unit element of $\Gamma$ does not commute with $u_{\alpha, \beta}$. By the same reason, we have $\langle w_{\beta} \rangle \cap \Gamma = \{ e \}$. Now we consider the subalgebra $A'_{\alpha,\beta} := \langle v_{\alpha}, w_{\beta} \rangle$.
    By using the $q$-binomial formula, we see that the generators satisfy
    \begin{equation*}
      w_{\beta} v_{\alpha} = q^2 v_{\alpha} w_{\beta}, \quad
      v_{\alpha}^N = \alpha^N \quad \text{and} \quad
      w_{\beta}^N = \beta^N.
    \end{equation*}
    This implies that $A'_{\alpha,\beta}$ is spanned by $\{ v_{\alpha}^k w_{\beta}^{\ell} \mid k, \ell \in [0, N) \}$.
    By using the basis \eqref{eq:uD-basis} of $H$, we see that this set is linearly independent over $\bfk \Gamma$. Therefore we have $A'_{\alpha,\beta} \cap \Gamma = \{ e \}$.
  \item Suppose that $G = \langle g^r \rangle$ for some $r$ with $r \mid N$ and $0 < r < N$.
    By our assumption that $N$ is odd, we have $1 \not \sim_G 2$.
    Furthermore, we have $\chi_i \not \equiv \varepsilon \pmod{G^{\perp}}$ ($i = 1, 2$).
    Thus a reduced datum $C_i$ ($i = 1, 2$) for the equivalence class $[i]$ is either of the form $(0 \ 0)$ or $(1 \ 0)$.
    Correspondingly to the case where $(C_1, C_2) = ((0 \ 0), (0 \ 0))$, $((1 \ 0), (0 \ 0))$, $((0 \ 0), (1 \ 0))$ and $((1 \ 0), (1 \ 0))$, the coideal subalgebra $A$ is either one of $\langle g^r \rangle$, $\langle g^r, x_1 \rangle$, $\langle g^r, x_2 \rangle$, $\langle g^r, x_1, x_2 \rangle$.
    By using the basis \eqref{eq:uD-basis} of $U(\mathcal{D})$, it is easy to check that the intersection of each of them and $\Gamma$ is equal to $\langle g^r \rangle$.
  \end{itemize}
  For $r = N$, we have $\langle g^r \rangle = \bfk$,
  $\langle g^r, x_1 \rangle = \langle v_0 \rangle$,
  $\langle g^r, x_2 \rangle = \langle w_0 \rangle$ and
  $\langle g^r, x_1, x_2 \rangle = \langle v_0, w_0 \rangle$.
  By summarizing the discussion so far, the coideal subalgebras of $H$ are
  \begin{gather*}
    \langle g^r \rangle, \quad \langle g^r, x_1 \rangle, \quad \langle g^r, x_2 \rangle, \quad \langle g^r, x_1, x_2 \rangle, \\
    \langle x_1 + \alpha x_2 + \beta g \rangle \quad ((\alpha, \beta) \ne (0,0)),
    \quad \langle x_2 + \beta g \rangle \quad (\beta \ne 0), \\
    \langle x_1 + \alpha g, x_2 + \beta g \rangle \quad ((\alpha, \beta) \ne (0,0)),
  \end{gather*}
  where $r$ is a positive divisor of $N$.
\end{example}

\subsection{A technique of partial derivatives}

We now introduce the main technique for proving Theorems \ref{thm:generators-1} and \ref{thm:generators-2}.
Till the end of this section, we fix a datum $\mathcal{D}$ as in \eqref{eq:datum} and define
\begin{equation*}
  \mathbb{I} = \{ (m_1, \cdots, m_{\theta}) \in \mathbb{Z}^{\theta} \mid
  \text{$m_i \in [0, N_i)$ ($i = 1, \cdots, \theta$)} \}.
\end{equation*}
An element of $\mathbb{I}$ is expressed by a bold face letter, such as $\boldsymbol{m}$, and the $i$-th entry of $\boldsymbol{m} \in \mathbb{I}$ is written as $m_i$ by using the corresponding letter.
We define
\begin{equation*}
  |\boldsymbol{m}| = m_1 + \cdots + m_{\theta}
  \quad \text{and} \quad
  x^{\boldsymbol{m}} = x_1^{m_1} \cdots x_{\theta}^{m_{\theta}}
\end{equation*}
for $\boldsymbol{m} \in \mathbb{I}$. An element of $U(\mathcal{D})$ of this form is called a {\em monomial}.
With this notation, the basis \eqref{eq:uD-basis} is written as $\{ g x^{\boldsymbol{m}} \mid g \in \Gamma, \boldsymbol{m} \in \mathbb{I} \}$.
By \eqref{eq:uD-comultiplication-1} and the $q$-binomial formula, the comultiplication is given by
\begin{equation}
  \label{eq:uD-comultiplication-2}
  \Delta(g x^{\boldsymbol{m}})
  = \sum_{\boldsymbol{i} + \boldsymbol{j} = \boldsymbol{m}}
  c(\boldsymbol{m}, \boldsymbol{i}, \boldsymbol{j}) \,
  g x^{\boldsymbol{i}} \otimes g g_1^{i_1} \cdots g_r^{i_r} x^{\boldsymbol{j}},
\end{equation}
where $\boldsymbol{i}$ and $\boldsymbol{j}$ run over all elements of $\mathbb{I}$ satisfying $\boldsymbol{i} + \boldsymbol{j} = \boldsymbol{m}$, and
\begin{equation*}
  c(\boldsymbol{m}, \boldsymbol{i}, \boldsymbol{j})
  = \prod_{r = 1}^{\theta} \binom{m_r}{i_r}_{\chi_r(g_r)}
  \cdot \prod_{r = 1}^{\theta} \prod_{s = 1}^{r-1} \chi_s(g_r)^{-i_r j_s}.
\end{equation*}

Mimicking techniques used in the study of Nichols algebras \cite{MR506406,MR1913436}, we introduce linear operators $\partial_i$ on $U(\mathcal{D})$ for each $i \in [1, \theta]$ behaving like partial derivation with respect to $x_i$.
For this purpose, we first consider the datum $\mathcal{D}_0$ obtained from $\mathcal{D}$ by changing scalar parameters $\lambda_{i j}$ and $\mu_i$ to zero.
For each $i \in [1, \theta]$, there is a 2-dimensional representation $\rho_i$ of $U(\mathcal{D}_0)$ given by
\begin{equation*}
  \rho_i(g) =
  \begin{pmatrix} \chi_i(g) & 0 \\ 0 & 1 \end{pmatrix}
  \quad (g \in \Gamma)
  \quad \text{and}
  \quad \rho_i(x_j) = \begin{pmatrix} 0 & 0 \\ \delta_{i j} & 0 \end{pmatrix}
  \quad (j = [1, \theta]),
\end{equation*}
where $\delta_{i j}$ is Kronecker's delta. For $u \in U(\mathcal{D}_0)$, we let $\alpha_i(u) \in \bfk$ and $\xi_i(u) \in \bfk$ be the (1,1) entry and the (2,1) entry of $\rho_i(u)$, respectively.
The multiplicativity of the map $\rho_i$ implies that the equations
\begin{equation}
  \label{eq:partial-der-eq-1}
  \alpha_i(u v) = \alpha_i(u) \alpha_i(v)
  \quad \text{and} \quad
  \xi_i(u v) = \xi_i(u) \alpha_i(v) + \varepsilon(u) \xi_i(v)
\end{equation}
hold for all $u, v \in U(\mathcal{D}_0)$.
Since $C = U(\mathcal{D}_0)$ is a right $C$-comodule, it has a natural structure of a left $C^*$-module as we have recalled in Subsection~\ref{subsec:convention}.
Using the left action of $C^*$, we define linear operators $\tau_i$ and $\partial_i$ on $U(\mathcal{D}_0)$ by
\begin{equation}
  \label{eq:partial-der-def}
  \tau_i(u) = \alpha_i \rightharpoonup u
  \quad \text{and} \quad
  \partial_i(u) = \xi_i \rightharpoonup u
\end{equation}
for $u \in U(\mathcal{D}_0)$, respectively. By \eqref{eq:partial-der-eq-1}, we have
\begin{equation}
  \label{eq:partial-der-eq-2}
  \tau_i(u v) = \tau_i(u) \tau_i(v)
  \quad \text{and} \quad
  \partial_i(u v) = \partial_i(u) \tau_i(v) + u \cdot \partial_i(v)
\end{equation}
for all $u, v \in U(\mathcal{D}_0)$.
The linear operator $\tau_i$ is an algebra automorphism of $U(\mathcal{D}_0)$ such that $\tau_i(g) = \chi_i(g) g$ for $g \in \Gamma$ and $\tau_i(x_j) = \chi_i(g_j) x_j$ for $j \in [1, \theta]$. The linear operator $\partial_i$ behaves like the partial derivation with respect to $x_i$ in the following sense:

\begin{lemma}
  \label{lem:partial-der-formula}
  For $g \in \Gamma$ and $\boldsymbol{m} \in \mathbb{I}$, we have
  \begin{equation}
    \label{eq:partial-der-eq-3}
    \partial_i(g x^{\boldsymbol{m}})
    = (m_i)_{q_i} g x_1^{m_1} \cdots x_{i}^{m_i-1} \cdots x_{\theta}^{m_{\theta}},
  \end{equation}
  where $q_i = \chi_i(g_i)$.
\end{lemma}
\begin{proof}
  By~\eqref{eq:partial-der-eq-2} and induction on $m$, we have
  \begin{equation*}
    \partial_i(x_i^m) = (m)_{q_i} x_i^{m - 1}
    \quad \text{and} \quad
    \partial_i(x_j^m) = 0 \quad (i \ne j).
  \end{equation*}
  We also have $\partial_i(a) = 0$ for all $a \in \bfk \Gamma$.
  Now \eqref{eq:partial-der-eq-3} follows from \eqref{eq:partial-der-eq-2}.
\end{proof}

By the formula \eqref{eq:uD-comultiplication-2} of the comultiplication, we see that the linear map
\begin{equation*}
  U(\mathcal{D}_0) \to U(\mathcal{D}),
  \quad g x^{\boldsymbol{m}}
  \mapsto g x^{\boldsymbol{m}}
  \quad (g \in \Gamma, \boldsymbol{m} \in \mathbb{I})
\end{equation*}
is an isomorphism of coalgebras.
We use this isomorphism to identify $U(\mathcal{D})$ with $U(\mathcal{D}_0)$ as a coalgebra and regard $\partial_i$ as a linear operator on $U(\mathcal{D})$.

\begin{lemma}
  \label{lem:partial-der-coideal}
  Every right coideal of $U(\mathcal{D})$ is stable under $\partial_i$.
\end{lemma}
\begin{proof}
  The coalgebra $C := U(\mathcal{D}_0)$ is a left module over $C^*$, and a right coideal of $C$ is nothing but a left $C^*$-submodule of $C$.
  The claim for $\mathcal{D}_0$ follows from that $\partial_i$ is defined by the left action of an element of $C^*$. The general case follows from that $U(\mathcal{D})$ is identified with $U(\mathcal{D}_0)$ as a coalgebra as explained in the above.
\end{proof}

\subsection{Proof of Theorem~\ref{thm:generators-1}}
\label{subsec:proof-generators-1}

We introduce the graded lexicographic order for monomials in $U(\mathcal{D})$.
Namely, given two elements $\boldsymbol{m}, \boldsymbol{n} \in \mathbb{I}$, we regard the monomial $x^{\boldsymbol{m}}$ as a monomial greater than $x^{\boldsymbol{n}}$ if either one of the following two conditions (1) or (2) holds:
\begin{itemize}
\item [(1)] $|\boldsymbol{m}| > |\boldsymbol{n}|$.
\item [(2)] $|\boldsymbol{m}| = |\boldsymbol{n}|$ and there exists an element $k \in [1, \theta]$ such that $m_{\ell} = n_{\ell}$ for all $\ell \in [1, k)$ and $m_k > n_k$.
\end{itemize}

Every element of $U(\mathcal{D})$ is expressed as a polynomial of the form $\sum_{\boldsymbol{m} \in \mathbb{I}} c_{\boldsymbol{m}} x^{\boldsymbol{m}}$ ($c_{\boldsymbol{m}} \in \bfk \Gamma$) in a unique way. For the discussion, we borrow the basic terminology for polynomials: If a non-zero element $y \in U(\mathcal{D})$ is expressed as $y = c x^{\boldsymbol{m}} + \mbox{}$(lower terms) for some $\boldsymbol{m} \in \mathbb{I}$ and $c \in \bfk \Gamma \setminus \{ 0 \}$, then we call $c$, $x^{\boldsymbol{m}}$, $c x^{\boldsymbol{m}}$ and $\boldsymbol{m}$ the leading coefficient, the leading monomial, the leading term and the degree of $y$, respectively.

Unlike the case of polynomials over a field, the leading coefficient of an element of $U(\mathcal{D})$ may not be invertible.
We provide Lemma \ref{lem:eg-action} below to remedy this difficulty.
For $g \in \Gamma$, we define $e_g \in U(\mathcal{D})^*$ by
\begin{equation}
  \label{eq:eg-def}
  e_g(h x^{\boldsymbol{m}})
  = \delta_{g,h} \delta_{m_1,0} \dotsb \delta_{m_{\theta},0}
  \quad (h \in \Gamma, \boldsymbol{m} \in \mathbb{I}),
\end{equation}
where $\delta$ is Kronecker's delta. It is straightforward to prove:

\begin{lemma}
  \label{lem:eg-action}
  For $h \in \Gamma$ and $\ell \in [1, \theta]$, we have
  \begin{equation}
    \label{eq:eg-action}
    e_g \rightharpoonup h = \delta_{g, h} h
    \quad \text{and} \quad
    e_g \rightharpoonup h x_{\ell}
    = \delta_{g, h g_{\ell}} h x_{\ell}.
  \end{equation}
\end{lemma}

By applying the technique of the previous subsection, we prove:

\begin{lemma}
  \label{lem:coideal-sub-generators}
  Let $A$ be a coideal subalgebra of $U(\mathcal{D})$, and let $y$ be a non-zero element of $A$ with leading coefficient $c$. Then the following hold:
  \begin{enumerate}
  \item [(a)] The leading coefficient $c$ of $y$ belongs to $A$.
  \item [(b)] Let $\boldsymbol{m} = (m_1, \cdots, m_{\theta})$ be the degree of $y$, and let $k \in [1, \theta]$. If $m_k > 0$, then there is an element of $A$ of the form
    \begin{equation*}
      x_k + \sum_{\ell = k + 1}^{\theta} \beta_{\ell} g_k g_{\ell}^{-1} x_{\ell} + \alpha g_k
      \quad (\beta_{k+1}, \cdots, \beta_{\theta}, \alpha \in \bfk).
    \end{equation*}
  \end{enumerate}
\end{lemma}
\begin{proof}
  Part (a) follows from Lemma \ref{lem:partial-der-coideal} and
  \begin{equation*}
    c = \frac{1}{(m_1)_{q_1}! \cdots (m_{\theta})_{q_{\theta}}!}
    \cdot \partial_1^{m_1} \cdots \partial_{\theta}^{m_{\theta}}(y).
  \end{equation*}
  We now assume that $m_k > 0$. By Lemma \ref{lem:partial-der-coideal}, the element
  \begin{equation*}
    u := \frac{(m_k)_{q_k}}{(m_1)_{q_1}! \cdots (m_{\theta})_{q_{\theta}}!}
    \cdot \partial_1^{m_1} \cdots \partial_k^{m_k-1} \cdots \partial_{\theta}^{m_{\theta}}(y)
  \end{equation*}
  belongs to $A$.
  By the definition of the order of monomials, the element $u$ is of the form $u = c x_k + \sum_{\ell = k + 1}^{\theta} b_{\ell} x_{\ell} + a$ for some $b_{k+1}, \cdots, b_{\theta}, a \in \bfk \Gamma$.
  Given $z \in \bfk \Gamma$ and $g \in \Gamma$, we denote by $z(g) \in \bfk$ the coefficient of $g$ in $z$ so that $z = \sum_{g \in \Gamma} z(g) g$.
  Since $c$ is a non-zero element of $A \cap \bfk \Gamma = \bfk G$, where $G = A \cap \Gamma$, there is an element $g \in G$ such that $c(g) \ne 0$. We fix such an element $g$.
  By \eqref{eq:eg-action}, we have
  \begin{equation*}
    e_{g g_k} \rightharpoonup u = c(g) g x_k +
    \sum_{\ell = k + 1}^{\theta} b_{\ell}(g g_k g_{\ell}^{-1}) g g_k g_{\ell}^{-1} x_{\ell}
    + a(g g_k) g g_k.
  \end{equation*}
  Letting $\beta_{\ell} = c(g)^{-1} b_{\ell}(g g_k g_{\ell}^{-1})$ and $\alpha = c(g)^{-1} a(g g_k)$, we have
  \begin{equation*}
    x_k + \sum_{\ell = k + 1}^{\theta} \beta_{\ell} g_k g_{\ell}^{-1} x_{\ell} + \alpha g_k
    = c(g)^{-1} g^{-1} (e_{g g_k} \rightharpoonup u) \in A.
  \end{equation*}
  The proof is done.
\end{proof}

\begin{proof}[Proof of Theorem~\ref{thm:generators-1}]
  Let $A$ be a coideal subalgebra of $U(\mathcal{D})$, and let $A'$ be the subalgebra of $A$ generated by the set \eqref{eq:coideal-sub-generators-1}.
  It is trivial that $A' \subset A$. The converse inclusion is proved by mathematical induction as follows:
  Let $y$ be a non-zero element of $A$ of degree $\boldsymbol{m}$.
  If $\boldsymbol{m} = \boldsymbol{0}$, then it is obvious that $y \in A'$.
  We assume that $\boldsymbol{m} \ne \boldsymbol{0}$ and that every non-zero element of $A$ whose degree is smaller than $\boldsymbol{m}$ belongs to $A'$.
  We define $\hat{x}_k$ ($k = 1, \cdots, \theta$) as follows:
  If $m_k = 0$, then we set $\hat{x}_k = 1$.
  Otherwise, we let $\hat{x}_k$ be an element of $A$ whose leading term is $x_k$, which exists by Lemma~\ref{lem:coideal-sub-generators}.
  By the definition of the monomial order, we see that the leading term of $\hat{x}_1^{m_1} \dotsb \hat{x}_{\theta}^{m_{\theta}}$ is $x^{\boldsymbol{m}}$.
  Let $c$ be the leading coefficient of $y$, which belongs to $A$ by Lemma~\ref{lem:coideal-sub-generators}.
  By the induction hypothesis, we have
  \begin{equation*}
    y - c \hat{x}_1^{m_1} \dotsb \hat{x}_{\theta}^{m_{\theta}} \in A'.
  \end{equation*}
  Since $c, \hat{x}_1, \cdots, \hat{x}_{\theta} \in A'$, we have $y \in A'$. The proof is done.
\end{proof}

\subsection{Proof of Theorem~\ref{thm:generators-2}}
\label{subsec:proof-generators-2}

We prove Theorem~\ref{thm:generators-2} by refining the choice of generators in the proof of Theorem~\ref{thm:generators-1} by a linear algebraic method.
We first improve Lemma \ref{lem:coideal-sub-generators} (b) as follows:

\begin{lemma}
  \label{lem:coideal-sub-generators-2}
  Let $A$ be a coideal subalgebra of $U(\mathcal{D})$, let $y$ be a non-zero element of $A$ with degree $\boldsymbol{m}$, and let $k \in [1, \theta]$ be any index.
  Set $G = A \cap \Gamma$. If $m_k > 0$, then there is an element $v \in A$ of the form
  \begin{equation*}
    v = x_k + \sum_{\ell = k + 1}^{\theta} \beta_{\ell}^{} g_k^{} g_{\ell}^{-1} x_{\ell}
    + \alpha g_k
    \quad (\beta_{k+1}^{}, \cdots, \beta_{\theta}^{}, \alpha \in \bfk)
  \end{equation*}
  such that the following two conditions are satisfied:
  \begin{enumerate}
  \item $\beta_{\ell}^{} = 0$ if $k \not \sim_G \ell$.
  \item $\alpha = 0$ if $\chi_k \not \equiv \varepsilon \pmod{G^{\perp}}$.
  \end{enumerate}
\end{lemma}
\begin{proof}
  By Lemma \ref{lem:coideal-sub-generators}, there is an element of $A$ of the form
  \begin{equation*}
    u = x_k + \sum_{\ell = k + 1}^{\theta} \beta_{\ell}^{} g_k^{} g_{\ell}^{-1} x_{\ell} + \alpha g_k
    \quad (\beta_{k+1}^{}, \cdots, \beta_{\theta}^{}, \alpha \in \bfk).
  \end{equation*}
  Since $\beta_{\ell}^{} g_k^{} g_{\ell}^{-1} = \partial_{\ell}(u) \in A \cap \bfk \Gamma = \bfk G$, we have $\beta_{\ell}^{} = 0$ if $g_k^{} \not \equiv g_{\ell}^{} \pmod{G}$. For simplicity of notation, we set $u_k = x_k$, $u_{\ell} = \beta_{\ell}^{} g_k^{} g_{\ell}^{-1} x_{\ell}$, $u_{\theta + 1} = \alpha g_k$ and $\chi_{\theta + 1} = \varepsilon$. Then we have $u = u_{k} + \cdots + u_{\theta + 1}$ and
  \begin{equation}
    \label{eq:coideal-sub-generators-2-proof-1}    
    g u g^{-1} = \chi_k(g) u_k + \dotsb + \chi_{\theta + 1}(g) u_{\theta + 1}
    \quad (g \in \Gamma).
  \end{equation}
  Now we define $v = |G|^{-1} \sum_{g \in G} \chi_k(g)^{-1} g u g^{-1}$.
  This is an element of $A$ since $u \in A$ and $G \subset A$.
  By~\eqref{eq:coideal-sub-generators-2-proof-1} and orthogonality of characters, we have
  \begin{equation*}
    v = \sum_{\ell = k}^{\theta + 1} \varepsilon_{\ell} u_{\ell},
    \quad \text{where} \ \varepsilon_{\ell} =
    \begin{cases}
      1 & \text{if $\chi_k \equiv \chi_{\ell}$ (mod $G^{\perp}$)}, \\
      0 & \text{otherwise},
    \end{cases}
  \end{equation*}
  and thus $v$ has the desired form. The proof is done.
\end{proof}

\begin{proof}[Proof of Theorem~\ref{thm:generators-2}]
  Let $A$ be a coideal subalgebra of $U(\mathcal{D})$. As in the statement of the theorem, we let $\mathfrak{J}_1, \cdots, \mathfrak{J}_m$ be the set of the equivalence classes of $\sim_G$, where $G = \Gamma \cap A$.
  For each $k = 1, \cdots, \theta$, we define $\hat{x}_k \in A$ as follows:
  \begin{itemize}
  \item [(1)] If there is a non-zero element of $A$ such that $x_k$ appears in its leading monomial, then we define $\hat{x}_k$ as follows:
    We write the equivalence class of $[k]$ as $[k] = \{ i_1, \dotsc, i_n \}$, where $i_1 < \dotsb < i_n$.
    By Lemma \ref{lem:coideal-sub-generators-2}, $A$ has an element of the form
    \begin{gather*}
      u := x_{k} + c_{1} g_{k}^{} g_{i_{r+1}}^{-1} x_{i_{r+1}} + \cdots + c_n g_{k}^{} g_{i_{n}}^{-1} x_{i_n} + a g_{k} \\
      (\text{$c_{1}, \cdots, c_n, a \in \bfk$ and
        $a = 0$ if $\chi_k \not \equiv \varepsilon$ (mod $G^{\perp}$)}),
    \end{gather*}
    where $r$ is the index such that $k = i_r$.
    We now set $\hat{x}_k = g_{i_1}^{} g_{k}^{-1} u$, which belongs to $A$ since $u \in A$ and $i_1 \sim_G k$.
  \item [(2)] If (1) is not the case, we set $\hat{x}_k = 1$.
  \end{itemize}
  We see that the set $G \cup \{ \hat{x}_1, \cdots, \hat{x}_{\theta} \}$ generates $A$ by the same argument as the proof of Theorem \ref{thm:generators-1}.
  Moreover, by the construction of $\hat{x}_1, \cdots, \hat{x}_k$, there is a row echelon matrix $C_s$ for each $s = 1, \cdots, m$ such that
  \begin{equation*}
    \{ \hat{x}_k \mid k \in \mathfrak{J}_s \} \setminus \{ 1 \}
    = X_{\mathfrak{J}_s}(C_s) \setminus \{ 0 \}
    \quad (s = 1, \cdots, m).
  \end{equation*}
  Thus the set $G \cup X_{\mathfrak{J}_1}(C_1) \cup \cdots \cup X_{\mathfrak{J}_m}(C_m)$ generates $A$. For each $s$, we can retake the matrix $C_s$ to satisfy (RD1)--(RD3) in the following way:
  \begin{itemize}
  \item [(1)] The condition (RD2) is satisfied by the construction of $\hat{x}_1, \cdots, \hat{x}_{\theta}$.
  \item [(2)] Since the space spanned by the set $X_{\mathfrak{J}_s}(C_s)$ does not change if we replace $C_s$ with $P C_s$ for some invertible matrix $P$, we may assume that $C_s$ is a matrix of reduced row echelon form by performing the Gaussian elimination.
    Note that, after this process, (RD2) is still satisfied.
  \item [(3)] We assume that (RD1) and (RD2) are satisfied in view of (1) and (2), and let $B_s$ be the submatrix of $C_s$ obtained by removing the last column. Suppose that $C_s$ violates (RD3). Then the $r$-th row of $C_s$, where $r$ is the rank of $B_s$, is $(0,0,\cdots,0,1)$. If we change the $(r,n+1)$-th entry of $C_s$ to zero, then the set $X_{\mathfrak{J}_s}(C_s)$ loses the element $g_{j_s}^{}$. However, this does not matter since $g_{j_s}^{}$ is already contained in $G$.
  \end{itemize}
  The proof is done.
\end{proof}

\section{Coideal subalgebras of $\overline{U}_q(\mathfrak{sl}_2)$}
\label{sec:CSA-Uq-sl2}

\subsection{The Hopf algebra $\overline{U}_q(\mathfrak{sl}_2)$}
\label{subsec:Uq-sl2}

In this section, $\bfk$ is an algebraically closed field of characteristic zero.
Let $N > 1$ be an odd integer, and let $q \in \bfk$ be a primitive $N$-th root of unity.
We define $\overline{U}_q(\mathfrak{sl}_2)$ to be the algebra generated by $K$, $E$ and $F$ subject to the defining relations
\begin{gather*}
  K^N = 1,
  \ K E = q^2 E K,
  \ K F = q^{-2} F K,
  \ E^N = F^N = 0,
  \ E F - F E = \frac{K - K^{-1}}{q - q^{-1}}.
\end{gather*}
The algebra $\overline{U}_q(\mathfrak{sl}_2)$ is a Hopf algebra with the comultiplication $\Delta$, the counit $\varepsilon$ and the antipode $S$ determined by
\begin{gather*}
  \Delta(K) = K \otimes K,
  \ \Delta(E) = E \otimes K + 1 \otimes E,
  \ \Delta(F) = F \otimes 1 + K^{-1} \otimes F, \\
  \ \varepsilon(K) = 1,
  \ \varepsilon(E) = \varepsilon(F) = 0,
  \ S(K) = K^{-1},
  \ S(E) = -E K^{-1},
  \ S(F) = -K F.
\end{gather*}

The Hopf algebra $\overline{U}_q(\mathfrak{sl}_2)$ is an example of the Hopf algebra $U(\mathcal{D})$ constructed from a datum \eqref{eq:datum}.
More precisely, we set $\theta = 2$, $\Gamma = \langle g \mid g^N = e \rangle$, $g_1 = g_2 = g$, $\chi_1(g) = q^2$, $\chi_2(g) = q^{-2}$, $\lambda_{1 2} = 1$ and $\mu_1 = \mu_2 = 0$. Then $U(\mathcal{D})$ is generated by the grouplike element $g$ and $(g,1)$-skew primitive elements $x_1$ and $x_2$ subject to
\begin{gather*}
  g^N = 1,
  \quad g x_1 = q^{2} x_1 g,
  \quad g x_2 = q^{-2} x_2 g, \\
  x_1^N = x_2^N = 0,
  \quad \text{and} \quad x_2 x_1 - q^2 x_1 x_2 = 1 - g^2.
\end{gather*}
Hence the Hopf algebra $U(\mathcal{D})$ is identified with $\overline{U}_q(\mathfrak{sl}_2)$ by putting
\begin{equation*}
  K = g, \quad E = x_1 \quad \text{and} \quad F = \frac{1}{q-q^{-1}} g^{-1} x_2.
\end{equation*}

\subsection{Coideal subalgebras of $\overline{U}_q(\mathfrak{sl}_2)$}

We classify coideal subalgebras of $\overline{U}_q(\mathfrak{sl}_2)$ with the help of Theorem \ref{thm:generators-2}. To simplify the computation, we introduce
\begin{equation*}
  \tilde{F} = (q-q^{-1}) K F.
\end{equation*}

\begin{theorem}
  \label{thm:CSA-Uq-sl-2}
  A complete list of coideal subalgebras of $\overline{U}_q(\mathfrak{sl}_2)$ is
  \begin{gather*}
    \overline{U}_q(\mathfrak{sl}_2),
    \quad \langle K^r \rangle,
    \quad \langle K^r, E \rangle,
    \quad \langle K^r, \tilde{F} \rangle
    \quad (\text{$r$ is a positive divisor of $N$}), \\
    \langle E + \alpha \tilde{F} + \beta K \rangle
    \quad (\alpha, \beta \in \bfk, (\alpha, \beta) \ne (0,0)),
    \quad \langle \tilde{F} + \beta K \rangle \quad (\beta \in \bfk, \beta \ne 0), \\
    \quad \langle E + \lambda K, \tilde{F} + \mu K \rangle
    \quad (\lambda, \mu \in \bfk, (1-q^2) \lambda \mu = 1).
  \end{gather*}
\end{theorem}

\begin{proof}
  We take $\mathcal{D}$ as in the previous subsection and identify $\overline{U}_q(\mathfrak{sl}_2)$ with $U(\mathcal{D})$.
  Then the generators of $U(\mathcal{D})$ are expressed by $g = K$, $x_1 = E$ and $x_2 = \tilde{F}$.
  Let $A$ be a coideal subalgebra of $U(\mathcal{D})$.
  \begin{itemize}
  \item We first consider the case where $A \cap \Gamma = \{ e \}$.
    By the same argument as Example~\ref{ex:CSA-gr-Uq-sl2}, we find that $A$ is either one of $\bfk$, $\langle u_{\alpha, \beta} \rangle$, $\langle w_{\beta} \rangle$ or $\langle v_{\lambda}, w_{\mu} \rangle$ for some $\alpha, \beta, \lambda, \mu \in \bfk$, where $u_{\alpha,\beta} := x_1 + \alpha x_2 + \beta g$, $v_{\lambda} = x_1 + \lambda g$ and $w_{\mu} = x_2 + \mu g$.
    Again by the same argument as Example~\ref{ex:CSA-gr-Uq-sl2}, we can verify that $A' \cap \Gamma = \{ e \}$ holds when $A'$ is either one of the subalgebras $\bfk$, $\langle u_{\alpha, \beta} \rangle$ or $\langle w_{\beta} \rangle$.
    The subalgebra $B_{\lambda,\mu} := \langle v_{\lambda}, w_{\mu} \rangle$ does not always satisfy $B_{\lambda,\mu} \cap \Gamma = \{ e \}$ unlike Example~\ref{ex:CSA-gr-Uq-sl2}.
    In the present setting, one has
    \begin{equation}
      \label{eq:ex-uq-sl2-eq-1}
      B_{\lambda,\mu} \cap \Gamma = \{ e \} \iff (1-q^2) \lambda \mu = 1.
    \end{equation}
    To see this, we note the following equation:
    \begin{equation}
      \label{eq:ex-uq-sl2-eq-2}
      w_{\mu} v_{\lambda} - q^2 v_{\lambda} w_{\mu}
      = 1 + \{ (1 - q^2) \lambda \mu - 1 \} g^2.
    \end{equation}
    If $(1-q^2) \lambda \mu = 1$, then we have $B_{\lambda,\mu} \cap \Gamma = \{ e \}$ by the same argument as in Example~\ref{ex:CSA-gr-Uq-sl2}.
    If $(1 - q^2) \lambda \mu \ne 1$, then we have $g^2 \in B_{\lambda,\mu}$ from \eqref{eq:ex-uq-sl2-eq-2} and, in particular, $B_{\lambda,\mu} \cap \Gamma \ne \{ e \}$.
    The proof of \eqref{eq:ex-uq-sl2-eq-1} is done.
  \item Next, we consider the case where $A \cap \Gamma = \langle g^r \rangle$ for some $r$ with $r \mid N$ and $0 < r < N$.
    By the same argument as Example~\ref{ex:CSA-gr-Uq-sl2}, we find that $A$ is either one of $\langle g^r \rangle$, $\langle g^r, x_1 \rangle$, $\langle g^r, x_2 \rangle$, $\langle g^r, x_1, x_2 \rangle$. It is easy to see that $A' \cap \Gamma = \langle g^r \rangle$ if $A' = \langle g^r, x_1 \rangle$ or $A' = \langle g^r, x_2 \rangle$.
    Since $x_1$ and $x_2$ generate $U(\mathcal{D})$ unlike the case of Example~\ref{ex:CSA-gr-Uq-sl2}, the subalgebra $A' = \langle g^r, x_1, x_2 \rangle$ satisfies $A' \cap \Gamma = \langle g^r \rangle$ if and only if $r = 1$.
  \end{itemize}
  If $r = N$, then we have $\langle g^r \rangle = \bfk$, $\langle g^r, x_1 \rangle = \langle u_{0,0} \rangle$ and $\langle g^r, x_2 \rangle = \langle w_{0} \rangle$.
  By summarizing the discussion so far, we have the list of coideal subalgebras as stated. The proof is done.
\end{proof}

The coideal subalgebra $\langle K^r \rangle$ is isomorphic to the group algebra of the cyclic group of order $N/r$. To describe the structure of other coideal subalgebras, it is convenient to introduce the following algebra:

\begin{definition}
  \label{def:Taft-like-algebra}
  For a positive integer $m$, a positive divisor $n$ of $m$ and a root of unity $\xi \in \bfk$ of order $n$, we define $T_{m,n}(\xi)$ to be the algebra generated by $\mathtt{x}$ and $\mathtt{g}$ subject to the relations $\mathtt{x}^m = 0$, $\mathtt{g}^n = 1$ and $\mathtt{g} \mathtt{x} = \xi \mathtt{x} \mathtt{g}$.
\end{definition}

The algebra $T_{N,N}(q)$ is what is called the Taft algebra in the study of Hopf algebras.
Many of coideal subalgebras of $\overline{U}_q(\mathfrak{sl}_2)$ are Taft like:
Indeed, it is easy to see that $\langle K^r, E \rangle$ and $\langle K^r, K F \rangle$ are isomorphic to $T_{N, N/r}(q^{2N/r})$ as algebras.
Later, it turns out that the coideal subalgebra $\langle E + \alpha K, \tilde{F} + \beta K \rangle$ with $(1-q^2) \alpha \beta = 1$ is isomorphic to $T_{N,N}(q^2)$ as an algebra (see Theorem~\ref{thm:CSA-v-lambda-w-mu-structure}).

Table~\ref{tab:CSA-Uq-sl2} shows the coideals subalgebras of $\overline{U}_q(\mathfrak{sl}_2)$ and their dimension.
The coideal subalgebra $\langle E + \alpha K, \tilde{F} + \beta K \rangle$ with $(1-q^2) \alpha \beta = 1$ will be discussed in Subsection \ref{subsec:CSA-v-alpha-w-beta}.
The coideal subalgebra generated by $u_{\alpha,\beta} = E + \alpha \tilde{F} + \beta K$ and the minimal polynomial of $u_{\alpha,\beta}$ will be discussed in Subsection~\ref{subsec:CSA-u-alpha-beta}.

Figure~\ref{fig:CSA-Uq-sl2} is a Hasse-like diagram of the lattice of coideal subalgebras of $\overline{U}_q(\mathfrak{sl}_2)$.
In the diagram, Hopf subalgebras are annotated by a star ($\bigstar$).
Coideal subalgebras of the same dimension are arranged at the same height.
The elements $v_{\lambda}$ and $w_{\mu}$ are those appeared in the proof of the theorem:
\begin{equation*}
  v_{\lambda} = E + \lambda K, \quad w_{\mu} = \tilde{F} + \mu K
  \quad (\lambda, \mu \in \bfk).
\end{equation*}
Horizontal double arrows ($\Rightarrow$) mean to specialize a parameter. The coideal subalgebra $\langle v_{\lambda} \rangle$ appears twice because of the visibility of the diagram.
The coideal subalgebra $\langle E + \alpha \tilde{F} + \beta K \rangle$ ($\alpha, \beta \in \bfk$) comes in the diagram as $\langle a v_{\lambda} + \mu w_{\mu} \rangle$ with appropriately chosen $a$, $b$, $\lambda$ and $\mu$ (see Subsection~\ref{subsec:CSA-u-alpha-beta}).

\begin{table}
  \centering
  \renewcommand{\arraystretch}{1.5}
  \begin{tabular}{ccclccl}
    $A$ & Parameters & $A \cap \Gamma$ & $\dim(A)$ & Remarks \\ \hline
    $\overline{U}_q(\mathfrak{sl}_2)$ & --- & $\Gamma$ & $N^3$ \\
    $\langle K^r \rangle$ & $r \mid N$ & $\langle K^r \rangle$ & $N/r$ \\
    $\langle K^r, E \rangle$ & $r \mid N$ & $\langle K^r \rangle$ & $N^2/r$ \\
    $\langle K^r, \tilde{F} \rangle$ & $r \mid N$ & $\langle K^r \rangle$ & $N^2/r$ \\
    $\langle E + \alpha K, \tilde{F} + \beta K \rangle$ & $(1-q^2) \alpha \beta = 1$ & $\{ e \}$ & $N^2$ & \S\ref{subsec:CSA-v-alpha-w-beta} \\
    $\langle E + \alpha \tilde{F} + \beta K \rangle$ & $(\alpha, \beta) \ne (0,0)$ & $\{ e \}$ & $N$ & \S\ref{subsec:CSA-u-alpha-beta} \\
    $\langle \tilde{F} + \beta K \rangle$ & $\beta \ne 0$ & $\{ e \}$ & $N$ \\
  \end{tabular}
  \bigskip
  \caption{Coideal subalgebras of $\overline{U}_q(\mathfrak{sl}_2)$}
  \label{tab:CSA-Uq-sl2}
\end{table}

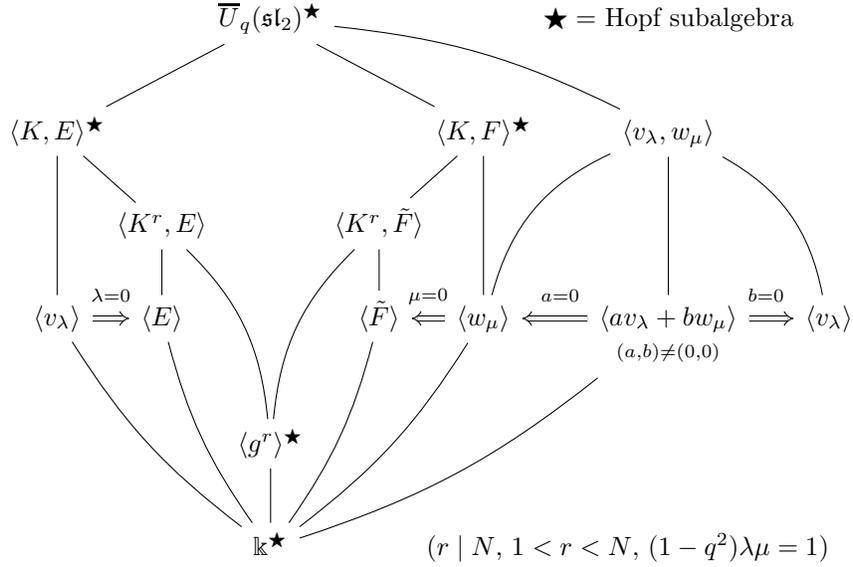
\begin{figure}
  \begin{equation*}
    \begin{tikzcd}[column sep = -4pt, row sep = 16pt]
      & & \overline{U}_q(\mathfrak{sl}_2) {}^{\bigstar}
      \arrow[lld, -] \arrow[rrd, -] \arrow[rrrd, -, bend left = 10]
      & & & \text{$\bigstar = $ Hopf subalgebra} & \\[8pt]
      \langle K, E \rangle {}^{\bigstar} \arrow[dd, -] \arrow[rd, -] & & & &
      \langle K, F \rangle {}^{\bigstar} \arrow[dd, -] \arrow[ld, -]
      & \langle v_{\lambda}, w_{\mu} \rangle
      \arrow[ddl, -, bend right = 25] \arrow[ddr, -, bend left = 25]
      \arrow[dd, -]
      & \\
      &   \langle K^r, E \rangle \arrow[d, -] \arrow[ddr, -, bend left = 20]
      & & \langle K^r, \tilde{F} \rangle \arrow[d, -] \arrow[ddl, -, bend right = 20]
      & & & \\
      \langle v_{\lambda} \rangle \arrow[ddrr, -, bend right = 10]
      \arrow[r, Rightarrow, "{\lambda = 0}" {yshift = 4pt}]
      & \langle E \rangle \arrow[ddr, -, bend right = 10]
      & & \langle \tilde{F} \rangle \arrow[ddl, -, bend left = 10]
      & \langle w_{\mu} \rangle \arrow[ddll, -, bend left = 10]
      \arrow[l, Rightarrow, "{\mu = 0}"' {yshift = 3pt}]
      & \begin{gathered}[t]
        \langle a v_{\lambda} + b w_{\mu} \rangle \\[-3pt]
        \scriptstyle (a, b) \ne (0,0) \end{gathered}
      \arrow[l, Rightarrow, "{a = 0}"' {yshift = 4pt, xshift = 2pt}]
      \arrow[r, Rightarrow, "{b = 0}" {yshift = 4pt, xshift = -2pt}]
      \arrow[ddlll, -, bend left = 10]
      & \langle v_{\lambda} \rangle \\[0pt]
      & & \langle g^r \rangle {}^{\bigstar} \arrow[d, -]
      & & & & \\[5pt]
      & & \bfk {}^{\bigstar} & & & & \makebox[0pt][r]{($r \mid N$, $1 < r < N$, $(1-q^2) \lambda \mu = 1$)}
    \end{tikzcd}
  \end{equation*}
  \caption{Lattice of coideal subalgebras of $\overline{U}_q(\mathfrak{sl}_2)$}
  \label{fig:CSA-Uq-sl2}
\end{figure}

\begin{remark}
  For a parameter $q \in \bfk^{\times}$ not being a root of unity, Vocke \cite{2018arXiv180410007V} gave a list of coideal subalgebras $A$ of $U_q(\mathfrak{sl}_2)$ such that $A \cap \Grp(U_q(\mathfrak{sl}_2))$ is a group.
  Although the setting we are considering in this paper is different from \cite{2018arXiv180410007V}, the conclusion is that the list of coideal subalgebras of $\overline{U}_q(\mathfrak{sl}_2)$ is basically same as Vocke's list.
\end{remark}

\begin{remark}
  Let $\mathcal{C}_m$ denote the set of all coideal subalgebras of $\overline{U}_q(\mathfrak{sl}_2)$ of dimension $m$.
  The sets $\mathcal{C}_N$ and $\mathcal{C}_{N^2}$ are parameterized neatly by using some geometric objects.
  Indeed, $\mathcal{C}_N$ can be identified with the projective plane by the map
  \begin{equation*}
    \mathbb{P}_2 \to \mathcal{C}_N,
    \quad [a : b : c] \mapsto \langle a E + b \tilde{F} + c K \rangle.
  \end{equation*}
  The set $\mathcal{C}_{N^2}$ can be identified with the projective line by the map
  \begin{equation*}
    \mathbb{P}_1 \to \mathcal{C}_{N^2},
    \quad [1 : 0] \mapsto \langle K, E \rangle,
    \quad [0 : 1] \mapsto \langle K, \tilde{F} \rangle,
    \quad [a : b] \mapsto \langle v_{\lambda}, w_{\mu} \rangle,
  \end{equation*}
  where $a, b \in \bfk^{\times}$, $\lambda = b/a$ and $\mu = (1-q^2)^{-1} a/b$.
\end{remark}

\begin{remark}
  Mombelli \cite[Section 8]{MR2678630} claimed that every exact comodule algebra over $\overline{U}_q := \overline{U}_q(\mathfrak{sl}_2)$ is obtained as a deformation of a coideal subalgebra, however, this is not true in general.
  In fact, when $N$ has a proper divisor $r$, Mombelli's list contains comodule algebras of dimension $r N^2$, while there are no such coideal subalgebras of $\overline{U}_q$ according to our classification result.
  This discrepancy originates from that Mombelli actually have classified exact comodule algebras over $\mathrm{gr}(\overline{U}_q)$, which are in bijection with those over $\overline{U}_q$.
  The Hopf algebra $\mathrm{gr}(\overline{U}_q)$ is identified with that considered in Example~\ref{ex:CSA-gr-Uq-sl2}.
  In view of the list of coideal subalgebras of $\mathrm{gr}(\overline{U}_q)$ given in Example~\ref{ex:CSA-gr-Uq-sl2}, we see that every exact comodule algebra over $\mathrm{gr}(\overline{U}_q)$ is indeed obtained as a deformation of a coideal subalgebra of $\mathrm{gr}(\overline{U}_q)$.
\end{remark}

\subsection{The coideal subalegbra \texorpdfstring{$\langle E + \lambda K, \tilde{F} + \mu K \rangle$}{<E+λK, \~F+βK>}}
\label{subsec:CSA-v-alpha-w-beta}

We fix elements $\lambda$ and $\mu$ of $\bfk$ satisfying $(1-q^2) \lambda \mu = 1$.
In this subsection, we discuss the structure of the coideal subalgebra $B_{\lambda,\mu} := \langle E + \lambda K, \tilde{F} + \mu K \rangle$.
For simplicity of notation, we use renormalized generators $v = \lambda^{-1} (E + \lambda K)$ and $w = \mu^{-1} (\tilde{F} + \mu K)$.
The main result in this section is that $B_{\lambda,\mu}$ is isomorphic to the algebra $T_{N,N}(q^2)$ of Definition~\ref{def:Taft-like-algebra}. More precisely, we prove:

\begin{theorem}
  \label{thm:CSA-v-lambda-w-mu-structure}
  There is the following isomorphism of algebras:
  \begin{equation*}
    T_{N,N}(q^2) \to B_{\lambda,\mu},
    \quad \mathtt{x} \mapsto 1 - w v,
    \quad \mathtt{g} \mapsto w.
  \end{equation*}
\end{theorem}

As the first step of the proof, we give a basis and determine defining relations of the algebra $B_{\lambda,\mu}$. It is straightforward to verify that $v$ and $w$ satisfy
\begin{equation}
  \label{eq:CSA-v-w-relations}
  v^N = w^N = 1 \quad \text{and} \quad w v - q^2 v w = 1 - q^2.
\end{equation}

\begin{lemma}
  \label{lem:CSA-v-w-relations}
  The coideal subalgebra $B_{\lambda,\mu}$ has the set
  \begin{equation}
    \label{eq:CSA-v-w-basis}
    \{ v^i w^j \mid i, j \in [0, N) \}
  \end{equation}
  as a basis and has \eqref{eq:CSA-v-w-relations} as defining relations.
\end{lemma}
\begin{proof}
  Since the leading monomial of $v^i w^j$ is $x_1^i x_2^j$, we see that the set \eqref{eq:CSA-v-w-basis} is linearly independent.
  By \eqref{eq:CSA-v-w-relations}, the set \eqref{eq:CSA-v-w-basis} also spans $B_{\lambda,\mu}$. Hence \eqref{eq:CSA-v-w-basis} is a basis of $B_{\lambda,\mu}$.
  
  To show that \eqref{eq:CSA-v-w-relations} is defining relations of $B_{\lambda,\mu}$, we define $\tilde{B}_{\lambda,\mu}$ to be the algebra generated by indeterminates $\tilde{v}$ and $\tilde{w}$ subject to the relations $\tilde{v}^N = 1$, $\tilde{w}^N = 1$ and $\tilde{w} \tilde{v} - q^2 \tilde{v} \tilde{w} = 1 - q^2$. There is a surjective algebra map $\phi : \tilde{B}_{\lambda,\mu} \to B_{\lambda,\mu}$ sending $\tilde{v}$ and $\tilde{w}$ to $v$ and $w$, respectively. Since $\tilde{B}_{\lambda,\mu}$ is spanned by the set $\{ \tilde{v}^i \tilde{w}^j \mid i, j \in [0, N) \}$, we have $\dim(\tilde{B}_{\lambda,\mu}) \le N^2 = \dim(B_{\lambda,\mu})$. Hence, by linear algebra, the algebra map $\phi$ is in fact an isomorphism. The proof is done.
\end{proof}

For $k \in \mathbb{Z}$, we set $e_k = \sum_{i = 0}^{N-1} q^{-2ik} v^i$ and $\tilde{e}_k = \sum_{i = 0}^{N-1} q^{-2ik} w^i$.

\begin{lemma}
  \label{lem:CSA-v-w-integral}
  For $k, \ell \in \mathbb{Z}$, we have
  \begin{equation}
    e_k \tilde{e}_{\ell} (1 - w v) = (1 - q^{2k+2\ell}) e_k \tilde{e}_{\ell - 1}.
  \end{equation}
\end{lemma}
\begin{proof}
  For a non-negative integer $i$, one can prove
  \begin{equation*}
    w^i v = (1 - q^{2i}) w^{i-1} + q^{2i} v w^i
  \end{equation*}
  by induction on $i$. Hence,
  \begin{align*}
    \tilde{e}_{\ell} v
    & = \sum_{i = 0}^{N-1} q^{-2i\ell} w^{i-1}
      - \sum_{i = 0}^{N-1} q^{-2i\ell + 2i} w^{i-1}
      + \sum_{i = 0}^{N-1} q^{-2i \ell + 2i} v w^{i} \\
    & = q^{-2\ell} \sum_{i = 0}^{N-1} q^{-2i \ell} w^{i}
      - q^{-2\ell} \sum_{i = 0}^{N-1} q^{-2i (\ell-1)} w^{i}
      + \sum_{i = 0}^{N-1} q^{-2i (\ell - 1)} v w^{i} \\
    & = q^{-2\ell} \tilde{e}_{\ell} - q^{-2\ell} \tilde{e}_{\ell-1} + v \tilde{e}_{\ell-1}.
  \end{align*}
  Since $e_k v = q^{2k} e_k$ and $\tilde{e}_{\ell} w = q^{2\ell} \tilde{e}_{\ell}$, we have
  \begin{align*}
    e_k \tilde{e}_{\ell} (1 - w v)
    & = e_k \tilde{e}_{\ell} - q^{2\ell} e_k (q^{-2\ell} \tilde{e}_{\ell} - q^{-2\ell} \tilde{e}_{\ell-1} + q^{2k} \tilde{e}_{\ell-1}) \\
    & = (1 - q^{2k+2\ell}) e_k \tilde{e}_{\ell-1}.
      \qedhere
  \end{align*}
\end{proof}

\begin{proof}[Proof of Theorem~\ref{thm:CSA-v-lambda-w-mu-structure}]
  Set $x = 1 - w v$. For all $k, \ell \in \mathbb{Z}$, we have
  \begin{equation*}
    e_k \tilde{e}_{\ell} x^N
    = \prod_{j = 0}^{N-1} (1 - q^{2k + 2\ell - 2j}) e_k \tilde{e}_{\ell-N}
    = 0
  \end{equation*}
  by the previous lemma.
  Since the set $\{ e_k \tilde{e}_{\ell} \mid k, \ell \in [0, N) \}$ is a basis of $B_{\lambda,\mu}$, the element $x^N$ is zero. We also have $w x = q^2 x w$.
  Hence there is a unique algebra map $\phi : T_{N,N}(q^2) \to B_{\lambda,\mu}$ such that $\phi(\mathtt{x}) = x$ and $\phi(\mathtt{g}) = w$. The generators $v$ and $w$ are contained in the image of $\phi$, and therefore $\phi$ is surjective. Since the source and the target of $\phi$ have the same dimension, the map $\phi$ is in fact an isomorphism. The proof is done.
\end{proof}

\begin{remark}
  \label{rem:CSA-v-w-integral}
  Lemma \ref{lem:CSA-v-w-integral} shows that the element $\Lambda = e_0 \tilde{e}_{0}$ satisfy $\Lambda x = 0 = \Lambda \varepsilon(x)$ and $\Lambda w = \Lambda = \Lambda \varepsilon(w)$. This means that $\Lambda$ is a right integral of $B_{\lambda,\mu}$.
\end{remark}

\subsection{The coideal subalegbra \texorpdfstring{$\langle E + \alpha \tilde{F} + \beta K \rangle$}{<E+α\~F+βK>}}
\label{subsec:CSA-u-alpha-beta}

In this subsection, we fix $\alpha, \beta \in \bfk$ and consider the coideal subalgebra $A_{\alpha, \beta} := \langle u_{\alpha, \beta} \rangle$, where $u_{\alpha,\beta} = E + \alpha \tilde{F} + \beta K$. Our result is summarized as follows:

\begin{theorem}
  \label{thm:CSA-u-alpha-beta-min-pol}
  The minimal polynomial of $u_{\alpha, \beta}$ is
  \begin{equation*}
    \phi_{\alpha, \beta}(X) = (X - \beta) \prod_{k = 1}^{(N-1)/2} \left( X^2 - \beta(q^{2k} + q^{-2k}) X + \beta^2 + \frac{(q^{2k}-q^{-2k})^2}{1 - q^2} \alpha \right),
  \end{equation*}
  and therefore $A_{\alpha, \beta} \cong \bfk[X]/(\phi_{\alpha,\beta}(X))$ as algebras.
  The algebra $A_{\alpha,\beta}$ is semisimple if and only if $D(\alpha, \beta) \ne 0$, where
  \begin{equation}
    D(\alpha, \beta) = \prod_{k = 0}^{N-1} D_k(\alpha, \beta),
    \quad D_k(\alpha, \beta) = \beta^2 - \frac{(q^{k} + q^{-k})^2}{1 - q^2} \alpha.
  \end{equation}
\end{theorem}

For a generic parameter $q$, the characteristic polynomial of the action of the element $a(K^{-1}-1) + b E K^{-1} + c F \in U_q(\mathfrak{sl}_2)$ ($a, b, c \in \bfk$) on an irreducible representation of $U_q(\mathfrak{sl}_2)$ has been computed in \cite[Theorem 1]{MR1710737} and \cite[Proposition 4.2]{MR1183481}. A similar problem is also addressed in connection with the study of $\imath$-canonical bases in \cite{BeW18}.
Indeed, for $\alpha = q/(q-q^{-1})$ and $\beta = 0$, the polynomial $\phi_{\alpha,\beta}(X)$ coincides with the formula of the $\imath$-divided power derived in \cite{BeW18}.
The above theorem may be obtained by a similar technique as \cite{BeW18,MR1710737,MR1183481}, but we present a totally different method in the below.

We go back to the case where $q$ is a root of unity of order $N$.
By the Chinese remainder theorem, the algebra $\bfk[X]/(\phi_{\alpha,\beta}(X))$ is semisimple if and only if $\phi_{\alpha,\beta}(X)$ has no multiple roots.
As the first step of the proof of this theorem, we examine when $\phi_{\alpha,\beta}(X)$ has multiple roots.
We fix $\xi \in \bfk$ such that
\begin{equation}
  \label{eq:CSA-u-def-xi}
  \xi^2 = \beta^2 - \frac{4 \alpha}{1-q^2}
\end{equation}
and then define
\begin{equation*}
  \omega_k = \frac{1}{2} \beta(q^{2k} + q^{-2k}) + \frac{1}{2} \xi (q^{2k}-q^{-2k})
\end{equation*}
for $k \in \mathbb{Z}$. Then we have
\begin{equation*}
  (X-\omega_k) (X-\omega_{-k})
  = X^2 - \beta(q^{2k} + q^{-2k}) X + \beta^2 + \frac{(q^{2k}-q^{-2k})^2}{1 - q^2} \alpha
\end{equation*}
for all $k \in \mathbb{Z}$.
Noting that $\omega_0 = \beta$ and $\omega_{a} = \omega_{b}$ if $a \equiv b \pmod{N}$, we find that the roots of the polynomial $\phi_{\alpha,\beta}(X)$ are $\omega_k$ for $k \in [0, N)$, including multiple roots.

\begin{lemma}
  \label{lem:CSA-u-alpha-beta-min-pol-multiple-roots}
  $\phi_{\alpha, \beta}(X)$ has multiple roots if and only if $D(\alpha, \beta) = 0$.
\end{lemma}
\begin{proof}
  We set $Q(a) = (1-q^{2a}) / (1 + q^{2a})$ for $a \in \mathbb{Z}$ (the denominator is non-zero since $q$ is a root of unity of odd order). Then we have
  \begin{gather}
    \label{eq:CSA-u-min-pol-multiple-roots-proof-1}
    2 q^{2i+2j} (\omega_i - \omega_j)
    = (q^{2i} - q^{2j}) (1 + q^{2i + 2j}) (\xi - Q(i + j) \beta), \\
    \label{eq:CSA-u-min-pol-multiple-roots-proof-2}
    (q^{k} + q^{-k})^2 (1 - Q(k)^2) = 4
    \quad (i, j, k \in \mathbb{Z}).
  \end{gather}
  Now we suppose that $\phi_{\alpha,\beta}(X)$ has multiple roots.
  Then there are two integers $i$ and $j$ such that $\omega_i = \omega_j$ and $i \not \equiv j \pmod{N}$. By \eqref{eq:CSA-u-min-pol-multiple-roots-proof-1}, we have $\xi = Q(i + j) \beta$.
  Letting $k \in [0, N)$ be an integer such that $k \equiv i + j \pmod{N}$, we have
  \begin{equation*}
    \frac{\alpha}{1 - q^2}
    \mathop{=}^{\eqref{eq:CSA-u-def-xi}}
    \frac{\beta^2 - \xi^2}{4}
    = \frac{\beta^2 - Q(k)^2 \beta^2}{4}
    \mathop{=}^{\eqref{eq:CSA-u-min-pol-multiple-roots-proof-2}}
    \frac{\beta^2}{(q^{k} + q^{-k})^2}
  \end{equation*}
  and therefore $D(\alpha, \beta) = 0$.

  To prove the converse, we suppose that $D(\alpha, \beta) = 0$.
  Then $D_k(\alpha, \beta) = 0$ for some $k \in [0, N)$.
  For this $k$, we have
  \begin{equation*}
    \xi^2 - Q(k)^2 \beta^2
    \mathop{=}^{\eqref{eq:CSA-u-def-xi}}
    (1 - Q(k)^2) \beta^2 - \frac{4\alpha}{1 - q^2}
    \mathop{=}^{\eqref{eq:CSA-u-min-pol-multiple-roots-proof-2}}
    \frac{4 D_k(\alpha, \beta)}{(q^{k} + q^{-k})^2} = 0
  \end{equation*}
  and therefore either $\xi = Q(k) \beta$ or $\xi = - Q(k) \beta$ holds.
  If the latter is the case, then we have $\xi = -Q(k) \beta = Q(N-k) \beta$.
  Hence we may assume $\xi = Q(k) \beta$ for some $k$.
  Then, by \eqref{eq:CSA-u-min-pol-multiple-roots-proof-1}, we have $\omega_i = \omega_j$ for all $i, j \in \mathbb{Z}$ with $i + j \equiv k \pmod{N}$. Hence $\phi_{\alpha, \beta}(X)$ has multiple roots. The proof is done.
\end{proof}

The next step is to determine the dimension of $A_{\alpha, \beta}$.

\begin{lemma}
  \label{lem:CSA-u-alpha-beta-dim}
  $\dim A_{\alpha, \beta} = N^2$.
\end{lemma}
\begin{proof}
  The leading term of every element of $A_{\alpha, \beta}$ is of the form $c x_1^k$ for some $c \in \bfk$ and $k \in [0, N)$. Indeed, let $y \in A_{\alpha, \beta}$ be a non-zero element with leading coefficient $c \in \bfk \Gamma$ and leading monomial $x_1^s x_2^t$ ($s, t \in [0, N)$). Since $\bfk \Gamma \cap A = \bfk 1$, we have $c \in \bfk$ by Lemma~\ref{lem:coideal-sub-generators}.
  We suppose $t > 0$.
  Then, by Lemma~\ref{lem:coideal-sub-generators}, there is an element of $A_{\alpha, \beta}$ of the form $z = x_2 + \gamma g$ ($\gamma \in \bfk$). The element $z$ does not commute with $u_{\alpha, \beta}$, which contradicts to the commutativity of $A_{\alpha, \beta}$. Thus $t = 0$, that is, the leading term of $y$ is $c x_1^s$.

  Set $u = u_{\alpha, \beta}$ for simplicity.
  By taking their leading terms into account, we see that the elements $u^k$ ($k \in [0, N)$) are linearly independent. Now let $y \in A_{\alpha, \beta}$ be a non-zero element. By the above discussion, the leading term of $y$ is of the form $c x_1^k$ for some $c \in \bfk$ and $k \in [0, N)$. The element $y - c u^k \in A_{\alpha, \beta}$ has the degree strictly smaller than $y$. Thus, by induction, we can express every element of $A_{\alpha, \beta}$ as a polynomial of $u$ of degree $<N$. In conclusion, the set $\{ u^k \mid k \in [0, N) \}$ is a basis of $A_{\alpha, \beta}$. The proof is done.
\end{proof}

\begin{proof}[Proof of Theorem~\ref{thm:CSA-u-alpha-beta-min-pol}]
  The proof is easy if $\alpha = 0$.
  Thus, from now on, we assume $\alpha \ne 0$.
  Letting $\lambda = (\beta - \xi)/2$ and $\mu = (\beta + \xi) / (2\alpha)$, we have
  \begin{equation*}
    u_{\alpha, \beta} = v_{\lambda} + \alpha w_{\mu},
    \quad \lambda \mu = (1-q^2)^{-1},
  \end{equation*}
  where $v_{\lambda} = E + \lambda K$ and $w_{\mu} = \tilde{F} + \mu K$.
  Now let $m(X)$ be the minimal polynomial of $u_{\alpha,\beta}$.
  By Theorem \ref{thm:CSA-v-lambda-w-mu-structure} (or, rather, by Lemma~\ref{lem:CSA-v-w-relations} used to prove that theorem), there are algebra maps $\mathrm{c}_k : \langle v_{\lambda}, w_{\mu} \rangle \to \bfk$ ($k \in \mathbb{Z}$) such that $\mathrm{c}_k(v_{\lambda}) = q^{2k} \lambda$ and $\mathrm{c}_k(w_{\mu}) = q^{-2k} \mu$. Since $\mathrm{c}_k(u_{\alpha,\beta}) = q^{2k} \lambda + q^{-2k} \alpha \mu = \omega_k$, the roots $\omega_0, \cdots, \omega_{N-1}$ of $\phi_{\alpha,\beta}(X)$ are also roots of $m(X)$.
  By Lemma~\ref{lem:CSA-u-alpha-beta-dim}, the degree of $m(X)$ is $N$.
  Hence we have arrived at the following conclusion: If $\phi_{\alpha, \beta}(X)$ has no multiple roots, then $\phi_{\alpha, \beta}(X)$ is the minimal polynomial of $u_{\alpha, \beta}$.

  To complete the proof, we discuss the general case where $\phi_{\alpha, \beta}(X)$ may have multiple roots.
  In general, a finite-dimensional algebra can be embedded into a matrix algebra by using its faithful representation.
  We embed $\overline{U}_q(\mathfrak{sl}_2)$ into a matrix algebra and regard the equation $\phi_{\alpha, \beta}(u_{\alpha, \beta}) = 0$ as a system of polynomial equations of $\alpha$ and $\beta$. By the above discussion and Lemma~\ref{lem:CSA-u-alpha-beta-min-pol-multiple-roots}, the equation $\phi_{\alpha, \beta}(u_{\alpha, \beta}) = 0$ holds if $(\alpha, \beta) \in \bfk^2$ does not satisfy the polynomial equation $D(\alpha,\beta) = 0$.
  Hence the equation $\phi_{\alpha, \beta}(u_{\alpha, \beta}) = 0$ actually holds for all $(\alpha, \beta) \in \bfk^2$.
  Since $\phi_{\alpha,\beta}(X)$ and $m(X)$ has the same degree, we conclude that $\phi_{\alpha, \beta}(X)$ is the minimal polynomial of $u_{\alpha, \beta}$, no matter whether $\phi_{\alpha, \beta}(X)$ has multiple roots. The proof is done.
\end{proof}

\section{Coideal subalgebras of $\overline{\mathcal{O}}_q(SL_2)$}
\label{sec:CSA-Oq-SL2}

\subsection{The Hopf algebra $\overline{\mathcal{O}}_q(SL_2)$}

Continuing from the previous section, $\bfk$ is assumed to be an algebraically closed field of characteristic zero.
Let $N > 1$ be an odd integer, and let $q \in \bfk$ be a primitive $N$-th root of unity.
The algebra $\overline{\mathcal{O}}_q(SL_2)$ is generated by $a$, $b$, $c$ and $d$ subject to the defining relations
\begin{gather*}
  ba = qab, \quad ca = qac, \quad bc = cb, \quad db = qbd, \quad dc = qcd, \\
  ad - q^{-1} bc = da - qbc = a^N = d^N = 1, \quad b^N = c^N = 0.
\end{gather*}
The algebra $\overline{\mathcal{O}}_q(SL_2)$ has the Hopf algebra structure determined by
\begin{equation*}
  \begin{pmatrix} \Delta(a) & \Delta(b) \\ \Delta(c) & \Delta(d) \end{pmatrix}
  = \begin{pmatrix} a & b \\ c & d \end{pmatrix} \otimes
  \begin{pmatrix} a & b \\ c & d \end{pmatrix},
\end{equation*}
which reads, {\it e.g.}, $\Delta(a) = a \otimes a + b \otimes c$.

By the invertibility of $a$ and the relation $ad-q^{-1}bc = 1$, one can also define $\overline{\mathcal{O}}_q(SL_2)$ as an algebra generated by $a$, $b$ and $c$ subject to the relations $ba = qab$, $ca = qac$, $bc = cb$, $b^N = c^N = 0$ and $a^N = 1$.
Hence the set
\begin{equation}
  \label{eq:Oq-SL2-basis-1}
  \{ a^i b^j c^k \mid i, j, k \in [0, N) \}
\end{equation}
is a basis of $\overline{\mathcal{O}}_q(SL_2)$. A similar argument shows that the set
\begin{equation}
  \label{eq:Oq-SL2-basis-2}
  \{ b^i c^j d^k \mid i, j, k \in [0, N) \}
\end{equation}
is also a basis of $\overline{\mathcal{O}}_q(SL_2)$.

It is known that the Hopf algebras $\overline{\mathcal{O}}_q(SL_2)$ and $\overline{U}_q(\mathfrak{sl}_2)$ are dual to each other (see \cite[Appendix A]{MR3926231} for the detail).
More specifically, there is a unique isomorphism $\phi : \overline{\mathcal{O}}_q(SL_2) \to \overline{U}_q(\mathfrak{sl}_2)^*$ of Hopf algebras determined by
\begin{equation*}
  \begin{pmatrix}
    \langle \phi(a), u \rangle & \langle \phi(b), u \rangle \\
    \langle \phi(c), u \rangle & \langle \phi(d), u \rangle \end{pmatrix} = \rho(u)
  \quad (u \in \overline{U}_q(\mathfrak{sl}_2)),
\end{equation*}
where $\rho$ is the 2-dimensional representation of $\overline{U}_q(\mathfrak{sl}_2)$ given by
\begin{equation*}
  \rho(E) = \begin{pmatrix} 0 & 1 \\ 0 & 0 \end{pmatrix}, \quad
  \rho(F) = \begin{pmatrix} 0 & 0 \\ 1 & 0 \end{pmatrix}, \quad \text{and} \quad
  \rho(K) = \begin{pmatrix} q & 0 \\ 0 & q^{-1} \end{pmatrix}.
\end{equation*}

The isomorphism $\phi$ induces a Hopf pairing $(f, u) = \langle \phi(f), u \rangle$.
For $f \in \overline{\mathcal{O}}_q(SL_2)$ and $u \in \overline{U}_q(\mathfrak{sl}_2)$,
we set $f \leftharpoonup u = (f_{(1)}, u) f_{(2)}$.
This action makes $\overline{\mathcal{O}}_q(SL_2)$ a right module over $\overline{U}_q(\mathfrak{sl}_2)$. By Lemma~\ref{lem:dual-coideal-sub-2}, we have
\begin{equation*}
  \mathcal{C}(\overline{\mathcal{O}}_q(SL_2)) = \{ A^{\dcsa} \mid A \in \mathcal{C}(\overline{U}_q(\mathfrak{sl}_2)) \},
\end{equation*}
where $A^{\dcsa} = \{ f \in \overline{\mathcal{O}}_q(SL_2) \mid \text{$f \leftharpoonup u = \varepsilon(u) f$ for all $u \in A$} \}$.
Table \ref{tab:CSA-Oq-SL2} is a complete list of coideal subalgebras of $\overline{\mathcal{O}}_q(SL_2)$.
Their dimensions can be computed by Table~\ref{tab:CSA-Uq-sl2} and the equation \eqref{eq:dual-coideal-sub-dim}.
The generators of $\langle K^r \rangle^{\dcsa}$, $\langle K^r, E \rangle^{\dcsa}$, $\langle K^r, \tilde{F} \rangle^{\dcsa}$, $\langle E + \alpha K \rangle^{\dcsa}$ and $\langle \tilde{F} + \beta K \rangle^{\dcsa}$ will be determined in \S\ref{subsec:CSA-Oq-SL2-generators}.
The generators of $\langle E + \alpha \tilde{F} + \beta K \rangle^{\dcsa}$ and $\langle E + \alpha K, \tilde{F} + \beta K \rangle^{\dcsa}$ will be discussed in \S\ref{subsec:CSA-Oq-SL2-generators-2} and \S\ref{subsec:CSA-Oq-SL2-generators-3}, respectively.

\begin{table}
  \centering
  \renewcommand{\arraystretch}{1.5}
  \begin{tabular}{ccclcl}
    $A$ & Parameters & $\dim(A)$ & generators \\ \hline
    $\overline{U}_q(\mathfrak{sl}_2)^{\dcsa}$ & --- & $1$ & --- \\
    $\langle K^r \rangle^{\dcsa}$ & $r \mid N$ & $r N^2$ & $a^{N/r}$, $a^{-1}b$, $a c$ \\
    $\langle K^r, E \rangle^{\dcsa}$ & $r \mid N$ & $r N$ & $d^{N/r}$, $c d^{-1}$ \\
    $\langle K^r, \tilde{F} \rangle^{\dcsa}$ & $r \mid N$ & $r N$ & $a^{N/r}$, $a^{-1} b$ \\
    $\langle E + \alpha K, \tilde{F} + \beta K \rangle^{\dcsa}$ & $(1-q^2) \alpha \beta = 1$ & $N$ & \S\ref{subsec:CSA-Oq-SL2-generators-3} \\
    $\langle E + \alpha \tilde{F} + \beta K \rangle^{\dcsa}$ & $(\alpha, \beta) \ne (0,0)$ & $N^2$ & \S\ref{subsec:CSA-Oq-SL2-generators-2} \\
    $\langle \tilde{F} + \beta K \rangle^{\dcsa}$ & $\beta \ne 0$ & $N^2$ & \S\ref{subsec:CSA-Oq-SL2-generators} \\
  \end{tabular}
  \bigskip
  \caption{Coideal subalgebras of $\overline{O}_q(SL_2)$}
  \label{tab:CSA-Oq-SL2}
\end{table}

\subsection{The action of $\overline{U}_q(\mathfrak{sl}_2)$ on $\overline{\mathcal{O}}_q(SL_2)$}
\label{subsec:Uq-sl2-action-on-Oq-SL2}

To find generators of coideal subalgebras of $\overline{\mathcal{O}}_q(SL_2)$, we compute the action of $\overline{U}_q(\mathfrak{sl}_2)$ on $\overline{\mathcal{O}}_q(SL_2)$.
By the definition of the action, we have
\begin{equation*}
  \begin{pmatrix}
    a \leftharpoonup u & b \leftharpoonup u \\
    c \leftharpoonup u & d \leftharpoonup u    
  \end{pmatrix}
  = \begin{pmatrix}
    (a, u) a + (b, u) c & (a, u) b + (b, u) d \\
    (c, u) a + (d, u) c & (c, u) b + (d, u) d
  \end{pmatrix}
  = \rho(u) \begin{pmatrix} a & b \\ c & d \end{pmatrix}
\end{equation*}
for $u \in \overline{U}_q(\mathfrak{sl}_2)$. In particular, the actions of the generators are given by
\begin{equation}
  \label{eq:actions-generators}
  \begin{aligned}
    a \leftharpoonup E & = c,
    & b \leftharpoonup E & = d,
    & c \leftharpoonup E & = 0,
    & d \leftharpoonup E & = 0, \\
    a \leftharpoonup F & = 0,
    & b \leftharpoonup F & = 0,
    & c \leftharpoonup F & = a,
    & d \leftharpoonup F & = b, \\
    a \leftharpoonup K & = q a,
    & b \leftharpoonup K & = q b,
    & c \leftharpoonup K & = \smash{q^{-1}} c,
    & d \leftharpoonup K & = \smash{q^{-1}} d.
  \end{aligned}
\end{equation}
In principal, one can compute the action by \eqref{eq:actions-generators} and the rule
\begin{equation*}
  (s t) \leftharpoonup u
  = (s \leftharpoonup u_{(1)}) (t \leftharpoonup u_{(2)})
  \quad (s, t \in \overline{\mathcal{O}}_q(SL_2), u \in \overline{U}_q(\mathfrak{sl}_2)),
\end{equation*}
which follows from that $(-,-)$ is a Hopf pairing.
Being inspired by the co-double bosonization construction of $\overline{\mathcal{O}}_q(SL_2)$ explained in \cite[Theorem 4.1]{MR3872859}, we introduce new generators
\begin{equation}
  \label{eq:PBW-generators}
  x = b d^{-1} \ (= b d^{N-1}), \quad y = c d \quad \text{and} \quad z = d^2
\end{equation}
of $\overline{\mathcal{O}}_q(SL_2)$.
By the relations $d b = q b d$ and $d c = q d c$, we have
\begin{equation}
  \label{eq:PBW-base-change}
  x^i y^j z^k = q^{\frac{1}{2}i(i-1) - \frac{1}{2}j(j-1) - i j} b^i c^j d^{-i+j+2k}
  \quad (i, j \in [0, N), k \in \mathbb{Z}).
\end{equation}
Since~\eqref{eq:Oq-SL2-basis-2} is a basis of $\overline{\mathcal{O}}_q(SL_2)$, the set
\begin{equation}
  \label{eq:Oq-SL2-basis-3}
  \{ x^i y^j z^k \mid i, j, k \in [0, N) \}
\end{equation}
is also a basis of $\overline{\mathcal{O}}_q(SL_2)$.
With respect to this basis, the action of $\overline{U}_q(\mathfrak{sl}_2)$ is given as follows:

\begin{lemma}
  \label{lem:actions-x-y-z}
  For $i, j \in [0, N)$ and $k \in \mathbb{Z}$, we have
  \begin{align}
    \label{lem:action-E}
    x^i y^j z^k \leftharpoonup E
    & = q^{-2(j+k)+1} (i)_{q^2} x^{i-1} y^{j} z^k, \\
    \label{lem:action-F}
    x^i y^j z^k \leftharpoonup F
    & = (-i + 2j + 2k)_{q^2} x^{i+1} y^j z^k
      + q^{-2i} (j)_{q^2} x^i y^{j-1} z^k, \\
    \label{lem:action-K}
    x^i y^j z^k \leftharpoonup K
    & = q^{2(i-j-k)} x^i y^j z^k.
  \end{align}
\end{lemma}

When $i = 0$, we interpret the right-hand side of \eqref{lem:action-E}, having $x^{-1}$, as zero. The second term of the right-hand side of \eqref{lem:action-F} is also interpreted as zero when $j = 0$.
The integer $-i+2j+2k$ appearing in \eqref{lem:action-F} may be negative.
For $m \in \mathbb{Z}$ and $t \in \bfk \setminus \{ 0, 1 \}$, we understand the symbol $(m)_{t}$ as $(m)_{t} = (1 - t^m)/(1 - t)$.

\begin{proof}
  Equation \eqref{lem:action-K} is easily verified by \eqref{eq:actions-generators} and that the action of $K$ is an algebra automorphism.
  For $s, t \in \overline{\mathcal{O}}_q(SL_2)$ with $t \leftharpoonup E = 0$, we have
  \begin{equation*}
    (s t) \leftharpoonup E = (s \leftharpoonup E) (t \leftharpoonup K)
    + (s \leftharpoonup 1) (t \leftharpoonup E)
    = (s \leftharpoonup E) (t \leftharpoonup K)
  \end{equation*}
  and, in a similar way, $(t s) \leftharpoonup E = t (s \leftharpoonup E)$.
  Now let $i$, $j$ and $\ell$ be non-negative integers.
  By \eqref{eq:actions-generators}, we have
  \begin{align*}
    & b^i c^j d^{\ell} \leftharpoonup E
      = (b^i \leftharpoonup E) (c^j \leftharpoonup K) (d^{\ell} \leftharpoonup K) \\
    & = \sum_{r = 0}^{i-1}
      \underbrace{(b \leftharpoonup 1) \cdots (b \leftharpoonup 1)}_{i - r - 1}
      (b \leftharpoonup E)
      \underbrace{(b \leftharpoonup K) \cdots (b \leftharpoonup K)}_{r}
      \mbox{} \cdot q^{-j} c^j \cdot q^{-\ell} d^{\ell} \\
    & = \sum_{r = 0}^{i - 1} b^{i - r - 1 - j - \ell} \cdot d \cdot q^r b^r \cdot c^j d^{\ell}
      = \sum_{r = 0}^{i - 1} q^{2r - \ell} b^{i-1} c^j d^{\ell+1}
      = q^{-\ell} (i)_{q^2} b^{i-1} c^j d^{\ell+1}.
  \end{align*}
  In a similar way as above, we also have $b^i \leftharpoonup F = 0$,
  \begin{align*}
    c^j \leftharpoonup F
    & = \sum_{r = 0}^{j-1}
      \underbrace{(c \leftharpoonup K^{-1}) \cdots (c \leftharpoonup K^{-1})}_{r}
    (c \leftharpoonup F)
      \underbrace{(c \leftharpoonup 1) \cdots (c \leftharpoonup 1)}_{j - r - 1} \\
    & = \sum_{r = 0}^{i-1} q^{r} c^{r} \cdot a \cdot c^{j - r - 1}
      = \sum_{r = 0}^{i-1} q^{-j+2r+1} c^{j - 1} a
      = q^{-j+1} (j)_{q^2} c^{j-1} a,
  \end{align*}
  and $d^{\ell} \leftharpoonup F = (\ell)_{q^2} b d^{\ell-1}$.
  Hence we have
  \begin{align*}
    b^i c^j d^{\ell} \leftharpoonup F
    & = (b^i \leftharpoonup K^{-1}) (c^j \leftharpoonup F) d^{\ell} + (b^i \leftharpoonup K^{-1}) (c^j \leftharpoonup K^{-1}) (d^{\ell} \leftharpoonup F) \\
    & = q^{-i} b^i \cdot q^{-j + 1} (j)_{q^2} c^{j-1} \cdot a d^{\ell}
      + q^{-i} b^i \cdot q^j (\ell)_{q^2} c^j \cdot b d^{\ell-1} \\
    & = q^{-i - j} b^i c^{j-1} \{ q (j)_{q^2} a d
      + q^{2j} (\ell)_{q^2} b c \} d^{\ell - 1} \\
    & = q^{-i - j} b^i c^{j-1} \{ q (j)_{q^2}
      + (j)_{q^2} b c + q^{2j}(\ell)_{q^2} b c \} d^{\ell - 1} \\
    & = q^{-i - j + 1} (j)_{q^2} b^i c^{j-1} d^{\ell-1}
      + q^{-i - j} (j + \ell)_{q^2} b^{i+1} c^j d^{\ell-1},
  \end{align*}
  where the relation $a d = 1 + q^{-1} b c$ was used at the fourth equality.
  Although we have assumed that $\ell$ is non-negative, the formulas
  \begin{align*}
    b^i c^j d^{\ell} \leftharpoonup E
    & = q^{-\ell} (i)_{q^2} b^{i-1} c^j d^{\ell+1}, \\
    b^i c^j d^{\ell} \leftharpoonup F
    & = q^{-i - j + 1} (j)_{q^2} b^i c^{j-1} d^{\ell-1}
      + q^{-i - j + 1} (j + \ell)_{q^2} b^{i+1} c^j d^{\ell-1}
  \end{align*}
  we have obtained are actually valid for all $\ell \in \mathbb{Z}$ since both sides depend only on the remainder of $\ell$ divided by $N$.
  Now the formulas \eqref{lem:action-E} and \eqref{lem:action-F} are obtained by the base change by \eqref{eq:PBW-base-change}.
\end{proof}

By this lemma, we have the following submodules:
\begin{equation}
  \label{eq:Uq-action-on-Oq-submod-Vk-def}
  V_k := \Span \{ x^i y^j z^k \mid i, j \in [0, N) \}
  \quad (k \in \mathbb{Z}).
\end{equation}

\begin{lemma}
  \label{lem:Uq-action-on-Oq-submod-mult}
  For all $k, \ell \in \mathbb{Z}$, we have $V_k \cdot V_{\ell} = V_{k+\ell}$.
\end{lemma}
\begin{proof}
  This lemma is easily verified by noting that the generators $x$, $y$ and $z$ commute up to non-zero scalar multiples.
\end{proof}

\begin{lemma}
  \label{lem:Uq-action-on-Oq-submod-generator}
  The submodule $V_k$ is generated by $x^{N-1} y^{N-1} z^k$.
\end{lemma}
\begin{proof}
  By Lemma~\ref{lem:actions-x-y-z}, we have
  \begin{equation*}
    x^{N-1} y^{N-1} z^k \leftharpoonup F^{N-1-j}E^{N-1-i}
    = \text{(a non-zero constant)} \cdot x^{i} y^{j} z^k
  \end{equation*}
  for $i, j, k \in [0, N)$. 
\end{proof}

\subsection{Changing $q$ with $q^{-1}$}

The inverse of $q$ is also a root of unity of order $N$.
The Hopf algebra $\overline{U}_{q^{-1}}(\mathfrak{sl}_2)$ is known to be isomorphic to $\overline{U}_q(\mathfrak{sl}_2)$. More precisely, there is a unique isomorphism $\sigma : \overline{U}_{q^{-1}}(\mathfrak{sl}_2) \to \overline{U}_q(\mathfrak{sl}_2)$ of Hopf algebras sending $E$, $F$ and $K$ to $K F$, $E K^{-1}$ and $K$, respectively \cite[\S3.1.2]{MR1492989}.

There is also a unique isomorphism $\check{\sigma} : \overline{\mathcal{O}}_{q}(SL_2) \to \overline{\mathcal{O}}_{q^{-1}}(SL_2)$ of Hopf algebras such that $\check{\sigma}(a) = d$, $\check{\sigma}(b) = q c$, $\check{\sigma}(c) = q^{-1} b$ and $\check{\sigma}(d) = a$.
This is dual to the map $\sigma$ in the following sense:
For a while, we denote the pairing between $\overline{U}_q(\mathfrak{sl}_2)$ and $\overline{\mathcal{O}}_q(SL_2)$ by $(-,-)_q$ with specifying the parameter $q$.

\begin{lemma}
  For all $u \in \overline{U}_{q^{-1}}(\mathfrak{sl}_2)$ and $f \in \overline{\mathcal{O}}_q(SL_2)$, we have
  \begin{equation*}
    (f, \sigma(u))_q = (\check{\sigma}(f), u)_{q^{-1}}.
  \end{equation*}
\end{lemma}
\begin{proof}
  It is routine to verify that this equation holds for the case where $u$ and $f$ are generators. The general case follows from that $\sigma$ and $\check{\sigma}$ are Hopf algebra maps and $(-,-)_{q}$ and $(-,-)_{q^{-1}}$ are Hopf pairings.
\end{proof}

\begin{lemma}
  \label{lem:CSA-Oq-SL2-with-q-inverse}
  For a coideal subalgebra $A$ of $\overline{U}_{q^{-1}}(\mathfrak{sl}_2)$, we have
  \begin{equation*}
    \sigma(A)^{\dcsa} = \check{\sigma}(A^{\dcsa}).
  \end{equation*}
\end{lemma}
\begin{proof}
  For $f \in \overline{\mathcal{O}}_q(SL_2)$ and $u \in \overline{U}_{q^{-1}}(\mathfrak{sl}_2)$, we have
  \begin{equation*}
    \check{\sigma}(f) \leftharpoonup u
    = (\check{\sigma}(f_{(1)}), u)_{q^{-1}} \check{\sigma}(f_{(2)})
    = (f_{(1)}, \sigma(u))_{q} \, \check{\sigma}(f_{(2)})
    = \check{\sigma}(f \leftharpoonup \sigma(u)).
  \end{equation*}
  This lemma follows from this equation and Lemma~\ref{lem:dual-coideal-sub-2}.
\end{proof}

\subsection{Generators of coideal subalgebras of $\overline{\mathcal{O}}_q(SL_2)$}
\label{subsec:CSA-Oq-SL2-generators}

We describe generators of most of coideal subalgebras of $\overline{\mathcal{O}}_q(SL_2)$.

\begin{theorem}
  \label{thm:CSA-Oq-SL2-generators}
  Let $\alpha, \beta \in \bfk$. Then we have
  \begin{align}
    \label{eq:CSA-Oq-SL2-1}
    \overline{U}_q(\mathfrak{sl}_2)^{\dcsa} & = \langle 1 \rangle, \\
    \label{eq:CSA-Oq-SL2-2}
    \langle K^r \rangle^{\dcsa} & = \langle a^{N/r}, a^{-1} b, a c \rangle = \langle d^{N/r}, b d, c d^{-1} \rangle \\
    \label{eq:CSA-Oq-SL2-3}
    \langle K^r, E \rangle^{\dcsa} & = \langle d^{N/r}, c d^{-1} \rangle, \\
    \label{eq:CSA-Oq-SL2-4}
    \langle K^r, \tilde{F} \rangle^{\dcsa} & = \langle a^{N/r}, a^{-1} b \rangle, \\
    \label{eq:CSA-Oq-SL2-5}
    \langle E + \alpha K \rangle^{\dcsa} & = \langle c d^{-1}, d^2 + (q-q^{-1}) \alpha b d \rangle, \\
    \label{eq:CSA-Oq-SL2-6}
    \langle \tilde{F} + \beta K \rangle^{\dcsa} & = \langle a^{-1} b, a^2 -q^2 \beta a c \rangle.
  \end{align}
  In particular, we have
  \begin{gather*}
    \langle K \rangle^{\dcsa} = \langle a^{-1} b, a c \rangle = \langle b d, c d^{-1} \rangle,
    \quad \langle E \rangle^{\dcsa} = \langle c, d \rangle,
    \quad \langle \tilde{F} \rangle^{\dcsa} = \langle a, b \rangle, \\
    \langle K, E \rangle^{\dcsa} = \langle c d^{-1} \rangle,
    \quad \langle K, \tilde{F} \rangle^{\dcsa} = \langle K, F \rangle^{\dcsa} = \langle a^{-1} b \rangle.
  \end{gather*}
\end{theorem}
\begin{proof}
  Equation \eqref{eq:CSA-Oq-SL2-1} is obvious.
  By \eqref{eq:actions-generators}, it is easy to verify that the elements $a^{N/r}$, $a^{-1} b$ and $a c$ belong to $\langle K^r \rangle^{\dcsa}$. With the use of the basis \eqref{eq:Oq-SL2-basis-1}, we see that the dimension of the subalgebra generated by $a^{N/r}$, $a^{-1} b$ and $a c$ is $r N^2$, which is the dimension of $\langle K^r \rangle^{\dcsa}$. Hence we have the first equality in \eqref{eq:CSA-Oq-SL2-2}. The second equality in \eqref{eq:CSA-Oq-SL2-2} is obtained in a similar manner.

  Next, we show \eqref{eq:CSA-Oq-SL2-5}.
  For simplicity, we set $v_{\alpha} = E + \alpha K$.
  By Lemma~\ref{lem:actions-x-y-z}, it is straightforward to check that the elements
  \begin{equation*}
    s := y z^{N-1} = c d^{-1} \quad \text{and} \quad
    t := z + (q-q^{-1}) \alpha x z = d^2 + (q-q^{-1}) \alpha b d
  \end{equation*}
  belong to $\langle v_{\alpha} \rangle^{\dcsa}$. Now we set $A := \langle s, t \rangle$.
  Since the dimension of $\langle v_{\alpha} \rangle^{\dcsa}$ is $N^2$, we have $\dim(A) \le N^2$.
  We consider the algebra map
  \begin{equation*}
    \theta : A \xrightarrow{\quad i \quad}
    \overline{\mathcal{O}}_q(SL_2)
    \xrightarrow{\quad \pi \quad}
    T := \overline{\mathcal{O}}_q(SL_2)/(x),
  \end{equation*}
  where $i$ is the inclusion map and $\pi$ is the quotient map.
  It is easy to see that $T$ is an algebra of dimension $N^2$ generated by $\overline{y} := \pi(y)$ and $\overline{z} := \pi(z)$ subject to the relations $\overline{y}{}^N = 0$, $\overline{z}{}^N = 1$ and $\overline{z} \, \overline{y} = q^2 \overline{y} \, \overline{z}$.
  Since $\overline{y} = \theta(st)$ and $\overline{z} = \theta(t)$, the algebra map $\theta$ is surjective. By linear algebra, we obtain $\dim(A) \ge \dim(T) = N^2$ and therefore $A = \langle v_{\alpha} \rangle^{\dcsa}$. We have verified \eqref{eq:CSA-Oq-SL2-5}.

  Equation \eqref{eq:CSA-Oq-SL2-6} can be derived from \eqref{eq:CSA-Oq-SL2-5} and Lemma~\ref{lem:CSA-Oq-SL2-with-q-inverse}.
  To explain this in more detail, we denote the generators of $\overline{U}_{q^{-1}}(\mathfrak{sl}_2)$ and $\overline{\mathcal{O}}_{q^{-1}}(SL_2)$ by $E'$, $F'$, etc.\ to distinguish the generators of $\overline{U}_{q}(\mathfrak{sl}_2)$ and $\overline{\mathcal{O}}_{q}(SL_2)$. By applying \eqref{eq:CSA-Oq-SL2-5} to the coideal subalgebra $\langle E' + \alpha K' \rangle$ of $\overline{U}_{q^{-1}}(\mathfrak{sl}_2)$, we have
  \begin{equation*}
    \langle E' + \alpha K' \rangle^{\dcsa}
    = \langle c' d'{}^{-1}, d'{}^2 + (q^{-1} - q) \alpha b' d' \rangle
  \end{equation*}
  as subalgebras of $\overline{\mathcal{O}}_{q^{-1}}(SL_2)$.
  Hence, letting $\alpha = (q-q^{-1})^{-1} \beta$, we have
  \begin{gather*}
    \langle \tilde{F} + \beta K \rangle^{\dcsa}
    = \langle F K + \alpha K \rangle^{\dcsa}
    = \langle \sigma(E' + \alpha K') \rangle^{\dcsa} \\
    = \check{\sigma}(\langle E' + \alpha K' \rangle^{\dcsa})
    = \check{\sigma}(\langle c' d' {}^{-1}, d'{}^2 + (q^{-1} - q) \alpha b' d' \rangle) \\
    = \langle q^{-1} b a^{-1}, a^2 + (q^{-1} - q) \alpha q c a \rangle
    = \langle a^{-1} b, a^2 - q^2 \beta a c \rangle.
  \end{gather*}

  Equation \eqref{eq:CSA-Oq-SL2-3} is verified as follows:
  By \eqref{eq:Masuoka-Skryabin-correspondence-3}, \eqref{eq:CSA-Oq-SL2-2} and \eqref{eq:CSA-Oq-SL2-5}, we have
  \begin{equation*}
    \langle K^r, E \rangle^{\dcsa}
    = \langle K^r \rangle^{\dcsa} \cap \langle E \rangle^{\dcsa}
    = \langle d^{N/r}, b d, c d^{-1} \rangle \cap \langle c, d \rangle
    \supset \langle d^{N/r}, c d^{-1} \rangle.
  \end{equation*}
  Since both $\langle K^r, E \rangle^{\dcsa}$ and $\langle d^{N/r}, c d^{-1} \rangle$ have the same dimension, the last `$\supset$' is in fact an equality. We have proved \eqref{eq:CSA-Oq-SL2-3}.
  Equation \eqref{eq:CSA-Oq-SL2-4} is derived from \eqref{eq:CSA-Oq-SL2-3} in a similar way as we have derived \eqref{eq:CSA-Oq-SL2-6} from \eqref{eq:CSA-Oq-SL2-5}.
\end{proof}

It is not difficult to find relations between the generators of coideal subalgebras described in Theorem~\ref{thm:CSA-Oq-SL2-generators}. In more detail, we have:
\begin{enumerate}
\item The generators $\mathit{s} := b d$,
  $\mathit{t} := c d^{-1}$ and $u = d^{N/r}$ of $\langle K^r \rangle^{\dcsa}$ satisfy
  \begin{equation}
    \label{eq:CSA-Oq-SL2-Kr-dual-defining-relations}
    \mathit{s}^N = \mathit{t}^N = 0, \quad
    \mathit{u}^r = 1, \quad
    \mathit{t} \mathit{s} = q^{-2} \mathit{s} \mathit{t}, \quad
    \mathit{u} \mathit{s} = q^{N/r} \mathit{s} \mathit{u}, \quad
    \mathit{u} \mathit{t} = q^{N/r} \mathit{t} \mathit{u}
  \end{equation}
  and these equations are defining relations of $\langle K^r \rangle^{\dcsa}$.
\item There is an isomorphism of algebras
  \begin{equation*}
    T_{N,r}(q^{N/r}) \to \langle K^{r}, E \rangle^{\dcsa},
    \quad \mathtt{g} \mapsto d^{N/r},
    \quad \mathtt{x} \mapsto c d^{-1}.
  \end{equation*}
\item There is an isomorphism of algebras
  \begin{equation*}
    T_{N,r}(q^{-N/r}) \to \langle K^{r}, \tilde{F} \rangle^{\dcsa},
    \quad \mathtt{g} \mapsto a^{N/r},
    \quad \mathtt{x} \mapsto a^{-1} b.
  \end{equation*}
\item There is an isomorphism of algebras
  \begin{equation*}
    T_{N,N}(q^{-2}) \to \langle E + \alpha K \rangle^{\dcsa},
    \quad \mathtt{g} \mapsto d^2 + (q-q^{-1}) \alpha b d,
    \quad \mathtt{x} \mapsto c d^{-1}.
  \end{equation*}
\item There is an isomorphism of algebras
  \begin{equation*}
    T_{N,N}(q^{2}) \to \langle \tilde{F} + \beta K \rangle^{\dcsa},
    \quad \mathtt{g} \mapsto a^2 -q^2 \beta a c,
    \quad \mathtt{x} \mapsto a^{-1} b.
  \end{equation*}
\end{enumerate}

\subsection{The coideal subalgebra \texorpdfstring{$\langle E + \alpha \tilde{F} + \beta K \rangle^{\dcsa}$}{<E+α\~F+βK>†}}
\label{subsec:CSA-Oq-SL2-generators-2}

We fix $\alpha, \beta \in \bfk$ and discuss the coideal subalgebra $\langle u \rangle^{\dcsa}$, where $u = E + \alpha \tilde{F} + \beta K$.
Let $\phi_{\alpha,\beta}(X)$ be the minimal polynomial of $u$ given in Theorem~\ref{thm:CSA-u-alpha-beta-min-pol}.
Since $\phi_{\alpha,\beta}(X)$ has $\beta$ as a root,
\begin{equation*}
  \psi(X) := (X - \beta)^{-1} \phi_{\alpha,\beta}(X)
\end{equation*}
is a polynomial. Now we define $\Lambda_{\alpha,\beta} := \psi(u)$,
\begin{equation}
  \label{eq:CSA-Oq-SL2-generator-2}
  y_{\alpha,\beta} := \frac{q^{-1}}{(N-1)_{q^2}!}
  (x^{N-1} y \leftharpoonup \Lambda_{\alpha,\beta}), \quad
  z_{\alpha,\beta} := \frac{q^{-1}}{(N-1)_{q^2}!}
  (x^{N-1} z \leftharpoonup \Lambda_{\alpha,\beta}),
\end{equation}
where $x$, $y$ and $z$ are given by \eqref{eq:PBW-generators}.

\begin{theorem}
  \label{thm:CSA-Oq-SL2-generators-2}
  There is the following isomorphism of algebras:
  \begin{equation*}
    T_{N,N}(q^{-2}) \to \langle u \rangle^{\dcsa},
    \quad \mathtt{x} \mapsto y_{\alpha,\beta},
    \quad \mathtt{g} \mapsto z_{\alpha,\beta}.
  \end{equation*}
\end{theorem}
\begin{proof}
  We first note that $\Lambda := \Lambda_{\alpha, \beta}$ is a non-zero integral of $\langle u \rangle$. Indeed, it is not zero by the minimality of the polynomial $\phi_{\alpha,\beta}(X)$. Moreover, we have
  \begin{equation*}
    \Lambda u
    = (u - \beta) \Lambda + \beta \Lambda
    = \phi_{\alpha,\beta}(u) + \varepsilon(u) \Lambda = \varepsilon(u) \Lambda.
  \end{equation*}

  Since $\psi(X)$ is a monic polynomial of degree $(N-1)$, the element $\Lambda$ is of the form $E^{N-1}$ + (lower terms), where the order of monomials used here is the same as that used in the classification of coideal subalgebras of $\overline{U}_q(\mathfrak{sl}_2)$.
  Now let $m$ be a monomial strictly smaller than $E^{N-1}$.
  Then, by Lemma \ref{lem:actions-x-y-z}, when $x^{N-1} y \leftharpoonup m$ is expressed as a linear combination of the basis \eqref{eq:Oq-SL2-basis-3}, only the terms of the form $x^i y^j z^k$ with $i > 0$ can have non-zero coefficients. This observation yields that $y_{\alpha,\beta}$ is of the form
  \begin{equation*}
    y_{\alpha,\beta}
    = \frac{q^{-1}}{(N-1)_{q^2}!} \Big( x^{N-1} y \leftharpoonup E^{N-1} + x^{N-1} y \leftharpoonup (\text{lower terms}) \Big)
    = y + x f
  \end{equation*}
  for some $f \in \overline{\mathcal{O}}_q(SL_2)$.
  The same argument shows that $z_{\alpha,\beta}$ is of the form $z_{\alpha,\beta} = z + x f'$ for some $f' \in \overline{O}_q(SL_2)$.

  By Lemma~\ref{lem:dual-coideal-sub-3}, both $y_{\alpha,\beta}$ and $z_{\alpha,\beta}$ belong to the coideal subalgebra $\langle u \rangle^{\dcsa}$. To show that they generate $\langle u \rangle^{\dcsa}$, we let $A_{\alpha,\beta}$ be the subalgebra generated by $y_{\alpha,\beta}$ and $z_{\alpha,\beta}$.
  Now we deploy an argument similar to that used to find generators of $\langle E + \alpha K \rangle^{\dcsa}$ in Theorem~\ref{thm:CSA-Oq-SL2-generators}.
  Namely, we consider the algebra map
  \begin{equation*}
    \theta : A_{\alpha,\beta}
    \xrightarrow{\quad i \quad} \overline{\mathcal{O}}_q(SL_2)
    \xrightarrow{\quad \pi \quad} T := \overline{\mathcal{O}}_q(SL_2)/(x),
  \end{equation*}
  where $i$ is the inclusion map and $\pi$ is the quotient map.
  The algebra $T$ is isomorphic to $T_{N,N}(q^{-2})$ by the algebra map sending $\overline{y} := \pi(y)$ and $\overline{z} := \pi(z)$ to $X$ and $G$, respectively. By the above argument, we have
  \begin{equation*}
    \theta(y_{\alpha,\beta}) = \pi(y + x f) = \overline{y}
    \quad \text{and} \quad
    \theta(z_{\alpha,\beta}) = \pi(z + x f') = \overline{z}.
  \end{equation*}
  This implies that $\theta$ is surjective. By considering the dimensions of $A_{\alpha,\beta}$, $\langle u \rangle^{\dcsa}$ and $T$, we conclude that $A_{\alpha,\beta}$ is equal to $\langle u \rangle^{\dagger}$ and $\theta$ is an isomorphism.
  By composing the isomorphism $T_{N,N}(q^{-2}) \cong T$ and $\theta^{-1}$, we obtain an isomorphism as stated in this theorem. The proof is done.
\end{proof}

\begin{remark}
  The bijectivity of the map $\theta$ used in the above proof implies that $y_{\alpha,\beta}$ and $z_{\alpha,\beta}$ are characterized as unique elements of $\langle E + \alpha \tilde{F} + \beta K \rangle^{\dcsa}$ of the form $y_{\alpha,\beta} = y + x f$ and $z_{\alpha,\beta} = z + x f'$ for some $f, f' \in \overline{\mathcal{O}}_q(SL_2)$.
  For $\alpha = 0$, one can verify that $y_{0,\beta}$ and $z_{0,\beta}$ are given by
  \begin{equation*}
    y_{0,\beta} = y + (1-q^{-2}) \beta x y
    \quad \text{and} \quad
    z_{0,\beta} = z + (q-q^{-1}) \beta x z,
  \end{equation*}
  respectively, with the help of Lemma~\ref{lem:actions-x-y-z}.
  In general, it seems to be difficult to express $y_{\alpha,\beta}$ and $z_{\alpha,\beta}$ as linear combinations of bases \eqref{eq:Oq-SL2-basis-1}, \eqref{eq:Oq-SL2-basis-2} or \eqref{eq:Oq-SL2-basis-3}.
\end{remark}

\subsection{The coideal subalgebra \texorpdfstring{$\langle E + \lambda K, \tilde{F} + \mu K \rangle^{\dcsa}$}{<E+λK, \~F+μK>†}}
\label{subsec:CSA-Oq-SL2-generators-3}

We fix elements $\lambda$ and $\mu$ of $\bfk$ satisfying $(1-q^2) \lambda \mu = 1$.
In this subsection, we discuss the coideal subalgebra of $\overline{\mathcal{O}}_q(SL_2)$ corresponding to $B_{\lambda,\mu} = \langle v_{\lambda}, w_{\mu} \rangle$, where $v_{\lambda} = E + \lambda K$ and $w_{\mu} = \tilde{F} + \mu K$. Our result is summarized as follows:

\begin{theorem}
  \label{thm:CSA-Oq-SL2-generators-3}
  Let $\Lambda$ be a non-zero right integral of $B_{\lambda,\mu}$ and set
  \begin{equation}
    \label{eq:CSA-Oq-SL2-generator-3}
    w = x^{N-1} y^{N-1} z \leftharpoonup \Lambda.
  \end{equation}
  The coideal subalgebra $(B_{\lambda,\mu})^{\dcsa}$ is generated by $w$, and the minimal polynomial of $w$ is $X^N - \varepsilon(w)^N$. Furthermore, we have $\varepsilon(w) \ne 0$ and thus $(B_{\lambda,\mu})^{\dcsa}$ is semisimple as an algebra.
\end{theorem}

A non-zero right integral of $B_{\lambda,\mu}$ has been given in Remark~\ref{rem:CSA-v-w-integral}. Like the case discussed in the previous subsection, it seems to be difficult to express the generator $w$ as linear combinations of bases \eqref{eq:Oq-SL2-basis-1}, \eqref{eq:Oq-SL2-basis-2} or \eqref{eq:Oq-SL2-basis-3}.

We begin the proof of this theorem by showing:

\begin{lemma}
  \label{lem:CSA-Oq-SL2-generators-3-lemma-1}
  There is an isomorphism $(B_{\lambda, \mu})^{\dcsa} \cong \bfk^{N}$ of algebras, where the right hand side is the product of $N$ copies of the algebra $\bfk$.
\end{lemma}
\begin{proof}
  We set $H = \overline{U}_q(\mathfrak{sl}_2)$ and $A = \langle v_{\lambda}, w_{\mu} \rangle$.
  By definition, $A^{\dcsa}$ is isomorphic to the dual algebra of the quotient coalgebra $H/A^{+}H$.
  Since the left $A$-module $H$ has the set $\{ 1, K, \cdots, K^{N-1} \}$ as a basis, and since the equation $a h = \varepsilon(a) h$ holds in $H/A^{+}H$ for $a \in A$ and $h \in H$, the vector space $H/A^{+}H$ is spanned by the classes of $1, K, \cdots, K^{N-1}$.
  Since the dimension of $H/A^{+}H$ is $N$, they in fact form a basis of $H/A^{+}H$.
  Therefore $H/A^{+}H$ is isomorphic to $\bfk \langle K \rangle$ as a coalgebra. Thus we have $A^{\dagger} \cong (H/A^{+}H)^* \cong (\bfk \langle K \rangle)^* \cong \bfk^N$ as algebras. The proof is done.
\end{proof}

Since $\Lambda$ is a right integral of $B_{\lambda,\mu}$, the element $w$ belongs to $(B_{\lambda,\mu})^{\dcsa}$. However, because of the complexity of the expression of $w$, it is not obvious whether $w$ is even a non-zero element. A key observation for showing $w \ne 0$ is:

\begin{lemma}
  \label{lem:CSA-Oq-SL2-generators-3-lemma-2}
  For each $k \in [0, N)$, we define the linear map
  \begin{equation*}
    \phi_k : \overline{U}_q(\mathfrak{sl}_2) \to V_k,
    \quad \phi_k(u) = x^{N-1} y^{N-1} z^{k} \leftharpoonup u,
  \end{equation*}
  where $V_k$ is the submodule of $\overline{\mathcal{O}}_q(SL_2)$ introduced in Subsection~\ref{subsec:Uq-sl2-action-on-Oq-SL2}.
  The restriction of $\phi_k$ to $B_{\lambda,\mu}$ is an isomorphism.
\end{lemma}
\begin{proof}
  We set $U := \overline{U}_q(\mathfrak{sl}_2)$ and $J := \Ker(\phi_k)$.
  Lemma~\ref{lem:Uq-action-on-Oq-submod-generator} says that $\phi_k$ is surjective.
  By Lemma~\ref{lem:actions-x-y-z}, we have $K - q^{2k} \in J$. Hence we have $K^{a} E^{b} F^{c} = q^{2 a k} E^{b} F^{c}$ in $U/J$ for all $a, b, c \in [0, N)$.
  Therefore $U/J$ is spanned by the set
  \begin{equation}
    \label{lem:CSA-Oq-SL2-generators-3-lemma-2-eq-1}
    \{ E^r F^s + J \mid r, s \in [0, N) \},
  \end{equation}
  which is in fact a basis of $U/J$ since $\dim(U/J) = \dim(V_k) = N^2$.

  We recall from Lemma~\ref{lem:CSA-v-w-relations} that $B_{\lambda,\mu}$ has the set $\mathcal{B} = \{ v_\lambda^i w_{\mu}^j \mid i, j \in [0, N) \}$ as a basis. The element $v_{\lambda}^i w_{\mu}^j$ is of the form $E^i F^j$ + (lower terms), where we have used the same monomial order as in Section~\ref{sec:CSA-Uq-sl2}.
  In view of the basis \eqref{lem:CSA-Oq-SL2-generators-3-lemma-2-eq-1} of $U/J$, we see that the image of $\mathcal{B}$ in $U/J$ is also a basis of $U/J$. This implies that the restriction of $\phi_k$ to $B_{\lambda,\mu}$ is in fact an isomorphism, as stated.
\end{proof}

\begin{proof}[Proof of Theorem~\ref{thm:CSA-Oq-SL2-generators-3}]
  The element $w = \phi_1(\Lambda) \in V_1$ is non-zero by Lemma~\ref{lem:CSA-Oq-SL2-generators-3-lemma-2}.
  Moreover, $w$ belongs to $A := (B_{\lambda,\mu})^{\dcsa}$ as $\Lambda$ is a right integral.
  Lemma \ref{lem:CSA-Oq-SL2-generators-3-lemma-1} implies that $A$ has no non-zero nilpotent elements.
  Thus $w^k \ne 0$ for all $k \ge 0$. Since $\overline{\mathcal{O}}_q(SL_2) =  V_0 \oplus \dotsb \oplus V_{N-1}$ and $w^k \in V_k$ by Lemma~\ref{lem:Uq-action-on-Oq-submod-mult}, the set $\mathcal{B} := \{ 1, w, w^2, \cdots, w^{N-1} \}$ is linearly independent. Since $\dim(A) = N$, the set $\mathcal{B}$ is in fact a basis of $A$. Summarizing, we have
  \begin{equation}
    \label{eq:CSA-Oq-SL2-generators-3-proof-eq-1}
    A = \Span\{ 1, w, \cdots, w^{N-1} \},
    \quad
    A \cap V_k = \Span\{ w^k \}.
  \end{equation}
  Since $w^N \ne 0$ and $w^N \in A \cap V_0$, we have $w^N = \gamma$ for some $\gamma \in \bfk^{\times}$.
  This implies that the minimal polynomial of $w$ is $X^N - \gamma$.
  By taking the counit of the both sides of $w^N = \gamma$, we have $\gamma = \varepsilon(w)^N$.
  Therefore $\varepsilon(w) \ne 0$. The proof is completed.
\end{proof}

\section{Remarks}
\label{sec:remarks}

\subsection{Normal coideal subalgebras of $\overline{U}_q(\mathfrak{sl}_2)$ and $\overline{\mathcal{O}}_q(SL_2)$}
\label{subsec:normal-CSA}

As in the previous section, $\bfk$ is an algebraically closed field of characteristic zero, $N > 1$ is an odd integer, and $q \in \bfk$ is a root of unity of order $N$.
In this section, we give miscellaneous remarks regarding coideal subalgebras of $\overline{U}_q(\mathfrak{sl}_2)$ and $\overline{\mathcal{O}}_q(SL_2)$.
We first discuss the normality of them.
Let $H$ be a Hopf algebra with antipode $S$.
The right adjoint action $\triangleleft$ of $H$ is defined by $x \triangleleft h = S(h_{(1)}) x h_{(2)}$ for $h, x \in H$.
We recall that a right coideal subalgebra $A$ of $H$ is said to be {\em normal} if $a \triangleleft h \in A$ for all $a \in A$ and $h \in H$ \cite[Definition 1.2]{MR1271619}.
It is obvious that $\bfk$ and $H$ are normal coideal subalgebras of $H$.
We call them {\em trivial} normal coideal subalgebras of $H$.

\begin{proposition}
  \label{prop:normal-CSA-Uq-sl2}
  $\overline{U}_q(\mathfrak{sl}_2)$ has no non-trivial normal coideal subalgebras.
\end{proposition}
\begin{proof}
  Letting $\tilde{F} = (q-q^{-1}) K F$, the adjoint action of $\overline{U}_q(\mathfrak{sl}_2)$ is given by
  \begin{equation*}
    \begin{aligned}
      K \triangleleft K & = K,
      & K \triangleleft E & = (1-q^{-2}) K E,
      & K \triangleleft \tilde{F} & = (1-q^{2}) K \tilde{F}, \\
      E \triangleleft K & = q^{-2} E,
      & E \triangleleft E & = (1-q^{-2}) E^2,
      & E \triangleleft \tilde{F} & = q^{-2} (K^2 - 1), \\
      \tilde{F} \triangleleft K & = q^2 F,
      & \tilde{F} \triangleleft E & = 1-K^2,
      & \tilde{F} \triangleleft \tilde{F} & = (1-q^2) \tilde{F}{}^2
    \end{aligned}
  \end{equation*}
  on generators $E$, $\tilde{F}$ and $K$. By using these formulas and our list of coideal subalgebras of $\overline{U}_q(\mathfrak{sl}_2)$ given in Theorem~\ref{thm:CSA-Uq-sl-2}, one can verify the claim straightforwardly.
\end{proof}

Takeuchi proved that the assignment $A \mapsto H A^{+}$ is a bijection from the set of normal coideal subalgebras $A$ of $H$ such that ${}_AH$ is faithfully flat to the set of Hopf ideals $I$ of $H$ such that $H$ is coflat over the quotient coalgebra $H/I$ \cite[Theorem 3.2]{MR1271619}.
Now we assume that $H$ is finite-dimensional.
By Takeuchi's result and Skryabin's freeness theorem, we see that the bijection \eqref{eq:Masuoka-Skryabin-correspondence} induces a bijection between the set of Hopf subalgebras of $H$ and the set of normal right coideal subalgebras of $H^*$.
An immediate consequence of Proposition \ref{prop:normal-CSA-Uq-sl2} is:

\begin{proposition}
  $\overline{\mathcal{O}}_q(SL_2)$ has no non-trivial Hopf subalgebras.
\end{proposition}

We also have:

\begin{proposition}
  The non-trivial normal coideal subalgebras of $\overline{\mathcal{O}}_q(SL_2)$ are
  \begin{equation*}
    \langle a^{-1} b \rangle,
    \quad \langle c d^{-1} \rangle \quad \text{and}
    \quad \langle a^{N/r}, a^{-1} b, a c \rangle = \langle d^{N/r}, b d, c d^{-1} \rangle,
  \end{equation*}
  where $r$ varies among all positive divisors of $N$ with $r < N$.
\end{proposition}
\begin{proof}
  The non-trivial Hopf subalgebras of $\overline{U}_q(\mathfrak{sl}_2)$ are $\langle E, K \rangle$, $\langle F, K \rangle$ and $\langle K^r \rangle$, where $r$ is a positive divisor of $N$ with $r < N$.
  Hence the non-trivial normal coideal subalgebras of $\overline{\mathcal{O}}_q(SL_2)$ are $\langle E, K \rangle^{\dcsa}$, $\langle F, K \rangle^{\dcsa}$ and $\langle K^r \rangle^{\dcsa}$.
  Theorem~\ref{thm:CSA-Oq-SL2-generators} completes the proof.
\end{proof}

\subsection{Orbits by Hopf automorphisms}
\label{subsec:orbits}

Let $H$ be a finite-dimensional Hopf algebra.
Although a coideal subalgebra of $H$ is an analog of subgroups in group theory, the set $\mathcal{C}(H)$ of coideal subalgebras of $H$ can be infinite.
The group of Hopf algebra automorphisms of $H$ naturally acts on $\mathcal{C}(H)$.
Let $\overline{\mathcal{C}}(H)$ denote the set of orbits.
Chirvasitu, Kasprzak and Szulim \cite[Section 5]{MR4125589} asked if the set $\overline{\mathcal{C}}(H)$ is finite.
The answer is negative in general, as the following proposition shows:

\begin{proposition}
  The set $\overline{\mathcal{C}}(H)$ is infinite if $H = \overline{U}_q(\mathfrak{sl}_2)$ or $H = \overline{\mathcal{O}}_q(SL_2)$.
\end{proposition}
\begin{proof}
  For a finite-dimensional Hopf algebra $H$, the bijection between $\mathcal{C}(H)$ and $\mathcal{C}(H^*)$ induces a bijection $\overline{\mathcal{C}}(H)$ and $\overline{\mathcal{C}}(H^*)$. Hence it suffices to show the claim of this proposition for $H = \overline{U}_q(\mathfrak{sl}_2)$.

  For $c \in \bfk^{\times}$, there is a Hopf algebra automorphism $\vartheta_{c}$ of $\overline{U}_q(\mathfrak{sl}_2)$ determined by $\vartheta_{c}(E) = c^{-1} E$, $\vartheta_{c}(F) = c F$ and $\vartheta_{c}(K) = K$. One can show that the assignment $c \mapsto \vartheta_{c}$ gives an isomorphism between the group $\bfk^{\times}$ and the group of Hopf algebra automorphisms of $\overline{U}_q(\mathfrak{sl}_2)$ in a similar way as \cite[\S3.1.2]{MR1492989}, where Hopf algebra automorphisms of $U_q(\mathfrak{sl}_2)$ are discussed. Since we have
  \begin{equation*}
    \vartheta_c(\langle E + \alpha \tilde{F} + \beta K \rangle)
    = \langle E + c^2 \alpha \tilde{F} + c \beta K \rangle
    \quad (\alpha, \beta \in \bfk),
  \end{equation*}
  the coideal subalgebras $\langle E + \alpha \tilde{F} + K \rangle$ ($\alpha \in \bfk$) belong to different orbits. The proof is done.
\end{proof}

\subsection{On Maschke's theorem for coideal subalgebras}

Let $H$ be a finite-dimensional Hopf algebra, and let $A$ be a coideal subalgebra.
A {\em total integral} of $A$ is a right integral $\Lambda \in A$ such that $\varepsilon(\Lambda) \ne 0$.
{\em Maschke's theorem for finite-dimensional Hopf algebras} \cite[\S2.2]{MR1243637} states that $H$ is semisimple if and only if $H$ has a total integral. It is natural to ask if the same holds for coideal subalgebras.
If $A$ is semisimple, then $A$ has a total integral since the counit $\varepsilon: A \to \bfk$ splits as a right $A$-module map.
Koppinen \cite{MR1199682} showed that $A$ is semisimple if it has a total integral and $S^2(A) \subset A$.
On the basis of our classification results of coideal subalgebras, we point out that the converse does not holds in general if we drop the condition $S^2(A) \subset A$. Namely,

\begin{proposition}
  Suppose that $\alpha$ and $\beta$ are elements of $\bfk$ such that $\beta \ne 0$ and $4 \alpha = (1 - q^2) \beta^2$. Then the coideal subalgebra $\langle E + \alpha \tilde{F} + \beta K \rangle$ of $\overline{U}_q(\mathfrak{sl}_2)$ is not semisimple but has a total integral.
\end{proposition}
\begin{proof}
  Under the assumption of this proposition, the minimal polynomial $\phi_{\alpha, \beta}(X)$ of $u := E + \alpha \tilde{F} + \beta K$ given in Theorem~\ref{thm:CSA-u-alpha-beta-min-pol} is decomposed as
  \begin{equation*}
    \phi_{\alpha, \beta}(X) = (X - \beta) \prod_{k = 1}^{(N-1)/2} \left( X - \frac{q^{2k} + q^{-2k}}{2} \beta \right)^2.
  \end{equation*}
  Since $\phi_{\alpha,\beta}(X)$ has multiple roots, the coideal subalgebra $\langle u \rangle$ is not semisimple.
  We consider the polynomial $\psi(X) = (X-\beta)^{-1} \phi_{\alpha, \beta}(X)$.
  As we have noted in the proof of Theorem~\ref{thm:CSA-Oq-SL2-generators-2}, the element $\Lambda := \psi(u)$ is an integral in $\langle u \rangle$. Moreover, since $\beta$ is a single root of $\phi_{\alpha,\beta}(X)$, we have $\varepsilon(\Lambda) = \psi(\beta) \ne 0$. Namely, $\Lambda$ is a total integral of $\langle u \rangle$. The proof is done.
\end{proof}

\subsection{Examples of simple algebras in $\Rep(\overline{U}_q(\mathfrak{sl}_2))$}

As in ordinary ring theory, a {\em simple algebra} in a tensor category $\mathcal{C}$ is defined to be an algebra $A$ in $\mathcal{C}$ being simple in the category of $A$-bimodules in $\mathcal{C}$.
Given a Hopf algebra $H$, we denote by $\Rep(H)$ the tensor category of finite-dimensional left $H$-modules.
By the result of Skryabin \cite[Theorem 6.1]{MR2286047}, we see that a coideal subalgebra of $\overline{\mathcal{O}}_q(SL_2)$ is a simple algebra in $\Rep(\overline{U}_q(\mathfrak{sl}_2))$ by the left action given by $u \rightharpoonup f = f_{(1)} (f_{(2)}, u)$ for $u \in \overline{U}_q(\mathfrak{sl}_2)$ and $f \in \overline{\mathcal{O}}_q(SL_2)$.

We describe some examples of simple algebras in $\Rep(\overline{U}_q(\mathfrak{sl}_2)$ obtained from coideal subalgebras of $\overline{\mathcal{O}}_q(SL_2)$.
The left action $\rightharpoonup$ is given by
\begin{align*}
    E \rightharpoonup a & = 0,
    & E \rightharpoonup b & = a,
    & E \rightharpoonup c & = 0,
    & E \rightharpoonup d & = c, \\
    F \rightharpoonup a & = b,
    & F \rightharpoonup b & = 0,
    & F \rightharpoonup c & = d,
    & F \rightharpoonup d & = 0, \\
    K \rightharpoonup a & = q a,
    & K \rightharpoonup b & = q^{-1} b,
    & K \rightharpoonup c & = q c,
    & K \rightharpoonup d & = q^{-1} d
\end{align*}
on the generators. By the same way as in Subsection~\ref{subsec:Uq-sl2-action-on-Oq-SL2}, we have
\begin{align}
  \label{lem:left-action-E}
  E \rightharpoonup b^i c^j d^k
  & = q^{-2i-k+1}(i+k)_{q^2} b^i c^{j+1} d^{k-1} + q^{-2i-k+2} (i)_{q^2} b^{i-1} c^j d^{k-1} \\
  \label{lem:left-action-F}
  F \rightharpoonup b^i c^j d^k
  & = q^{i-j+1} (j)_{q^2} b^i c^{j-1} d^{k+1}, \\
  \label{lem:left-action-K}
  K \rightharpoonup b^i c^j d^k
  & = q^{-i+j-k} b^i c^j d^k
\end{align}
for all $i, j \in [0, N)$ and $k \in \mathbb{Z}$.

\begin{example}
  \label{ex:simple-alg-in-Rep-Uq-sl2-1}
  We choose a positive divisor $r$ of $N$ and consider the coideal subalgebra $A := \langle K^r \rangle^{\dcsa}$ of $\overline{\mathcal{O}}_q(SL_2)$.
  As we have seen in Subsection \ref{subsec:CSA-Oq-SL2-generators}, the algebra $A$
  is generated by $\mathit{s} := b d$, $\mathit{t} := c d^{-1}$ and $u = d^{N/r}$ subject to \eqref{eq:CSA-Oq-SL2-Kr-dual-defining-relations}.
  By \eqref{lem:left-action-E}--\eqref{lem:left-action-K}, the actions of $E$, $F$ and $K$ are given by
  \begin{align*}
    E \rightharpoonup s & = q^{-2} (q+q^{-1}) s t + q^{-1},
    & F \rightharpoonup s & = 0,
    & K \rightharpoonup s & = q^{-2} s, \\
    E \rightharpoonup t & = -q t^2,
    & F \rightharpoonup t & = 1,
    & K \rightharpoonup t & = q^2 t, \\
    E \rightharpoonup u & = q^{-N/r+1} (N/r)_{q^2} t u,
    & F \rightharpoonup u & = 0,
    & K \rightharpoonup u & = q^{-N/r} u.
  \end{align*}
\end{example}

\begin{example}
  The coideal subalgebra $\langle E + \alpha K \rangle^{\dcsa}$ ($\alpha \in \bfk$) is generated by $t := c d^{-1}$ and $v := d^2 + (q-q^{-1}) \alpha b d$ subject to $t^N = 0$, $v^N = 1$ and $v t = q^{2} t v$. The actions of $E$, $F$ and $K$ on $t$ are given in Example~\ref{ex:simple-alg-in-Rep-Uq-sl2-1}.
  By \eqref{lem:left-action-E}--\eqref{lem:left-action-K}, we have $F \rightharpoonup v = 0$, $K \rightharpoonup v = q^{-2} v$ and
  \begin{align*}
    E \rightharpoonup v
    & = q^{-1} (2)_{q^2} c d + (q-q^{-1}) \alpha (q^{-2} (2)_{q^2} b c + q^{-1}) \\
    & = (q+q^{-1}) t v + \alpha(1-q^{-2}).
  \end{align*}
\end{example}

\begin{example}
  We consider the coideal subalgebra $A := \langle E + \lambda K, \tilde{F} + \mu K \rangle^{\dcsa}$, where $\lambda$ and $\mu$ are elements of $\bfk$ satisfying $(1-q^2) \lambda \mu = 1$.
  By Theorem~\ref{thm:CSA-Oq-SL2-generators-3}, the algebra $A$ is generated by $w := h \leftharpoonup \Lambda$, where $h = b^{N-1} c^{N-1} d^2$ and $\Lambda$ is a non-zero right integral of $\langle E + \lambda K, \tilde{F} + \mu K \rangle$.
  We set $\delta = (h, \Lambda E)$. Then we have $\delta \ne 0$ and
  \begin{equation*}
    K \rightharpoonup w = q^{-2} w,
    \quad E \rightharpoonup w = \delta,
    \quad F \rightharpoonup w = -q \delta^{-1} w^2.
  \end{equation*}
  The first equation is easily verified as follows:
  \begin{equation*}
    K \rightharpoonup w
    = (K \rightharpoonup b^{N-1} c^{N-1} d^2) \leftharpoonup \Lambda
    \mathop{=}^{\eqref{lem:left-action-K}} q^{-2}w.
  \end{equation*}
  The second and the third equation can be obtained with the help of a result of Montgomery and Schneider \cite{MR1869767}. To explain this in more detail, we recall that $\overline{\mathcal{O}}_q(SL_2)$ has a right $\overline{U}_q(\mathfrak{sl}_2)$-submodule $V_k$ given by \eqref{eq:Uq-action-on-Oq-submod-Vk-def}. In terms of the generators $a$, $b$, $c$ and $d$ of $\overline{\mathcal{O}}_q(SL_2)$, it is expressed by
  \begin{equation*}
    V_k = \Span \{ b^i c^j d^{- i + j + 2k} \mid i, j \in [0, N) \}.
  \end{equation*}
  By \eqref{lem:left-action-E} and \eqref{lem:left-action-F}, we have
  \begin{align*}
    E \rightharpoonup w
    & = \text{(constant)} \cdot b^{N-2} c^{N-1} d^1 \leftharpoonup \Lambda
      \in A \cap V_0, \\
    F \rightharpoonup w
    & = \text{(constant)} \cdot b^{N-1} c^{N-2} d^3 \leftharpoonup \Lambda
      \in A \cap V_2.
  \end{align*}
  These equations, together with \eqref{eq:CSA-Oq-SL2-generators-3-proof-eq-1}, imply that $E \rightharpoonup w = \delta'$ and $F \rightharpoonup w = \delta'' w^2$ for some $\delta', \delta'' \in \bfk$. The scalar $\delta'$ is identical to $\delta$. Indeed,
  \begin{equation*}
    \delta'
    = \varepsilon(E \rightharpoonup w)
    = \varepsilon(E \rightharpoonup h \leftharpoonup \Lambda)
    = (h_{(1)}, \Lambda) \varepsilon(h_{(2)}) (h_{(3)}, E)
    = (h, \Lambda E) = \delta.
  \end{equation*}  
  According to \cite[Theorem 3.1]{MR1869767}, $\delta$ is invertible and $\delta'' = -q \delta^{-1}$.
\end{example}

\section*{Acknowledgment}

The authors thank the referee for careful reading of the manuscript.
The first author (K.\ S.) is supported by JSPS KAKENHI Grant Number 24K06676.
The second author (R.\ S.) is supported by JST SPRING Grant Number JPMJSP2124.

\def\cprime{$'$}
\let\origbibitem\bibitem
\renewcommand{\bibitem}[2][]{\origbibitem{#2}}

\end{document}